\documentclass[12 pt]{article}
 \usepackage[backref,colorlinks,linkcolor=red,anchorcolor=green,citecolor=blue]{hyperref}
\usepackage{amsfonts,amssymb}
\usepackage{amsmath}
\usepackage{amssymb}
\usepackage{amsthm}
\usepackage{graphicx}
\usepackage{cite}
\usepackage{verbatim}
\usepackage{enumerate}
\DeclareMathOperator*{\argmin}{arg\,min}
\usepackage{bm}
\providecommand{\keywords}[1]{\textbf{\textit{Keywords---}} #1}

\newtheorem{theorem}{Theorem}
  \newtheorem{lemma}[theorem]{Lemma}
  \newtheorem{corollary}[theorem]{Corollary}
  
  \newtheorem{definition}[theorem]{Definition}
   \newtheorem{thm}{\noindent Theorem}[section]
  \newtheorem{remark}[theorem]{Remark}
  \newtheorem{assumption}[thm]{Assumption}

\usepackage[margin = 1in]{geometry}
\begin{document}
 \title{\sc{Parametric holomorphy of elliptic eigenvalue problems}
}


          \author{Byeong-Ho Bahn\thanks{University of Massachusetts Amherst, (bban@umass.edu).}
         }

         \pagestyle{myheadings} \markboth{$(\bm{b},\varepsilon)-$Holomorphy of EVPs}{Byeong-Ho Bahn} \maketitle

          \begin{abstract}
          The study of parameter-dependent partial differential equations (parametric PDEs) with countably many parameters has been actively studied for the last few decades. In particular, it has been well known that a certain type of parametric holomorphy of the PDE solutions allows the application of deep neural networks without encountering the curse of dimensionality. This paper aims to propose a general framework for verifying the desired parametric holomorphy by utilizing the bounds on parametric derivatives. The framework is illustrated with examples of parametric elliptic eigenvalue problems (EVPs), encompassing both linear and semilinear cases. As the results, it will be shown that the ground eigenpairs have the desired holomorphy. Furthermore, under the same conditions, the bounds for the mixed derivatives of the ground eigenpairs are derived. These bounds are well known to take a crucial role in the error analysis of quasi-Monte Carlo methods. 
\end{abstract}


\section{Introduction}
\label{intro}
\subsection{Motivation and the goal}
Parametric partial differential equations (PDEs) have been studied tremendously throughout the last few decades. In this paper, parametric PDEs of interest, \eqref{parapdeex}, have their coefficient functions, such as the potential function of Schr\"odinger equation, affinely depending on countably many real-valued parameters as the form of \eqref{paradep}. Because of the parametric dependence of the coefficient functions, the solution depends on the parameters as well. One of the main studies regarding this kind of PDEs has been simultaneously solving it instead of solving it separately for different parameters. Such attempts can be found, for example, in \cite{schbest, para3, semi2013, Cohen2015, para7, sch4}. One of the known sufficient conditions needed to allow the simultaneously solving PDEs is a type of holomorphic dependence of the solution on the parameters (parametric holomorphy) called $(\bm{b},\varepsilon)-$ holomorphy (Definition \ref{behol}). The study of $(\bm{b},\varepsilon)$-holomorphy appears in many places, for example, in \cite{mishra, zech, para3, Cohen2015, deepuq, schbest, semi2013}.

 The holomorphy of PDE solutions guarantees various numerical applications to the PDE. One of the important properties of $(\bm{b},\varepsilon)$-holomorphy is that it allows simultaneously solving PDE problems in terms of approximation. It has been studied actively under the name of best $N$-term approximation. In \cite{para3}, more general frameworks for studying best $N$-term approximations are provided using $(\bm{b},\varepsilon)$-holomorphy (defined as holomorphic assumption \textbf{HA}$(p,\varepsilon)$ in Definition 2.1 in \cite{para3}). One of the well-known results from the paper is that any $(\bm{b},\varepsilon)-$holomophic function with $\bm{b}\in \ell^p$ and $p\in (0,1)$ admits best $N$-term approximations of the function. According to the result, it is possible to universally approximate the solution to a PDE with the error rate $N^{\frac{1}{p}-1}$ where $N$ is the number of polynomial terms used to approximate. In other words, it can be said that any $(\bm{b},\varepsilon)-$holomorphic function admits so-called Taylor generalized polynomial chaos (gpc) expansion (Proposition 2.3 in \cite{deepuq}). With this result, instead of solving the PDE with different parameters several times, it is possible to find a universal approximation and plug the different parameters to get the solutions, and it can be seen as solving the PDE simultaneously. A remarkable point to observe is that, even if the infinite number of parameters, it still has the convergent error rate which overcomes the curse of dimensionality.  
 
 The fact that any $(\bm{b},\varepsilon)$-holomorphic mapping admits the universal approximation has been used for proving the theoretical guarantees of various numerical applications. One of the notable results is the application of deep neural networks (DNNs). With the result of best $N$-term approximation, recently, \cite{zech} showed that any $(\bm{b},\varepsilon)$-holomorphic function can be expressed by DNNs without the curse of dimensionality. Furthermore, the results were applied to the Bayesian PDE inversion problem in \cite{zech}.
   
  The expressive power of DNNs for $(\bm{b},\varepsilon)$-holomorphic function has been used in  \cite{mishra} to show that the holomorphy property of the PDE solutions allows Deep Operator Neural Networks (DeepOnets) to solve the corresponding class of PDEs. Specifically, a DeepOnet consists of three essential parts: encoder, DNNs, and decoder. DeepOnet accepts infinite dimensional input, which can be coefficient functions or a domain of PDE. Then, the encoder transforms this infinite-dimensional input into a finite-dimensional object. Then, by using DNNs, it approximates the solution in finite dimensional space. Then, the decoder transforms the finite-dimensional object into an infinite-dimensional object, which would be the solution to the PDE. The result of dimension-independent expressive power in \cite{zech} is used in the DNN part of DeepOnet in \cite{mishra}. So far, various PDE problems have been shown to have $(\bm{b},\varepsilon)-$holomorphy. Some of the famous examples are presented in the section 4 of \cite{mishra}.
  
 The parametric holomorphy of PDE solutions also appears in the study of quasi-Monte Carlo (QMC) methods, for example, in \cite{quasi, dick, para5, alex2019, alex20201, para6, para8, sch1, para9}. In this area, estimating the bound for mixed derivatives is one of the important tasks to theoretically guarantee the convergent error rate. In recent years, parametric EVPs have been studied actively in \cite{andreev2010, kien2022, alex2019, alexey} and semilinear EVPs in \cite{Bahn2024} and obtained well enough bound to guarantee the QMC error convergence rates. One of the main questions in this paper is if the bound for mixed derivatives can be used to show the eigenpairs are $(\bm{b},\varepsilon)$-holomorphic. As an interesting result, the bounds for mixed derivatives can be recovered in this paper avoiding dealing with multi-index techniques.

 The main difficulty of showing the holomorphy of the ground eigenpairs of EVPs compared to other PDEs is that the definition of ground eigenvalue relies on the ordering of real number system. In other words, because complex number system does not have an ordering, the meaning of the smallest complex eigenvalue has not been well defined. At the same time, the first and the second smallest eigenvalues can cross in the complex plane which might cause non-holomorphy at the crossing point. In \cite{andreev2010}, the domain of holomorphy of ground eigenpair was restricted to the region where the two eigenvalues do not cross. Another difficulty is that some important properties, such as upper bound or parametric regularity of eigenvalue, strongly depend on the corresponding  ground state. Thus, knowing the information of one of them requires the information of another, so separate treatments for eigenvalue and ground state do not work well. In order to overcome these difficulties, this paper will adopt Banach space version of the implicit function theorem. This theorem makes it possible to track the ground eigenpair when it is perturbed from real space to complex space. For example, this technique has been used in \cite{Bahn2024}. Also, general discussion on the regularity of implicit map with the bound of mixed derivatives has been discussed in \cite{implicit2024}. With this analysis, in this paper, the desired region of holomorphy will be obtained without considering the crossing issue. Furthermore, it enables the simultaneous treatment of eigenvalue and ground state, resolving the issue.

 \subsection{Contribution}
 The contribution of this paper is twofold. Firstly, this paper introduces a concrete and general way to check $(\bm{b},\varepsilon)-$holomorphy of the solutions to parametric PDEs. It should be emphasized that the upper bound estimation of single variable derivatives (for example, can be obtained from the bounds in \cite{alexey, alex2019, Bahn2024}, but with a simpler form) is used with some additional conditions to obtain the holomorphy. As will be shown, this result implies that the ground eigenpairs of parametric EVPs, such as the linear case in \cite{alexey} including the semilinear case in \cite{Bahn2024}, are $(\bm{b},\varepsilon)-$holomorphic. This will provide the theoretical guarantee of dimension-independent DNN application to EVPs.
 
 The $(\bm{b},\varepsilon)$-holomorphy of the ground eigenpair of linear EVPs has been stated in Section 2.4.2 and Section 5.2.3 of \cite{zech}. In those sections, the holomorphy shown in Theorem 4 and Corollary 2 of \cite{andreev2010} has been used. In that case, the parametric coefficient enters only to diffusion coefficient (the function $A$ in \eqref{problem1}). The analysis used in \cite{andreev2010} mainly follows from the perturbation theory of Kato in \cite{kato}, focusing on the domain where the first and the second eigenvalues do not cross. This current paper suggests different analysis in more general settings, allowing the coefficient functions $B$ and $C$ in \eqref{problem1} also depend on the parameters. To be specific, this paper takes into account all three possible parametric coefficient functions of linear elliptic EVPs. This enables DeepOnet to have three coefficient input functions. In addition, the analysis in this paper makes the argument simpler, not taking into account the crossing of the eigenvalues, broadening the domain of holomorphy.

 Secondly, under the very same condition in which the first result assumes, the upper bound estimations of the mixed derivatives of ground eigenpairs obtained in \cite{alexey, Bahn2024, alex2019} are recovered without dealing with multi-index techniques. For example, in \cite{alexey}, the bound in \cite{alex2019} which has the form of $|(\bm{\nu}|!)^{1+\varepsilon}\bm{\beta}^{\bm{\nu}}$ was improved to $(|\bm{\nu}|!)\bm{\beta}$ by using the tool of falling factorial. Using same techniques with \cite{alexey}, the bound of the form of $(|\bm{\nu}|!)\bm{\beta}^{\bm{\nu}}$ of ground eigenpair of semilinear EVP was obtained in \cite{Bahn2024}. In this paper, a slightly weaker version of the aforementioned bounds are obtained with simpler computation.

\subsection{Organization}
In section 2, the definitions and notations related to $(\bm{b},\varepsilon)-$holmorphy will be provided. In section 3, the main result, a general framework to check $(\bm{b},\varepsilon)-$holomorphy, will be introduced. Also, the bound of mixed derivatives for the functions will be illustrated. Lastly, in section 4, utilizing the main result obtained in section 3, the $(\bm{b},\varepsilon)-$holomorphy of the ground eigenpairs of linear and semilinear EVPs will be presented with the bound for the mixed derivatives.  

\subsection{Notation}
\label{notation} Let $U:=[-1,1]^{\mathbb{N}}$ is the space of parameters where $\mathbb{N}$ is the set of natural numbers without $0$.  Throughout this paper, the dependence on the stochastic parameter $\bm{y}\in U$ has more importance than that of the spatial variable $x\in \Omega$. Thus, for convenience, the dependence on $x$ will mostly drop. In other words, instead of $u(x,\bm{y})$, the notation, $u(\bm{y})$, will be used. When the dependence on $\bm{y}$ is clear from the context without any confusion, it will sometimes be dropped as well. Also, unless otherwise stated, the notation for the spatial domain, $\Omega$, will be easily dropped. For example, $L^2$, $\int u$ and $H_0^1$ will be used instead of $L^2(\Omega)$, $\int_\Omega u(x) dx$ and $H_0^1(\Omega)$. Here, $L^r$ for $r\geq 1$ and $H_0^1$ is the space of Lebesgue measurable functions with finite $\|\cdot\|_{L^r}$ and $\|\cdot\|_{H_0^1}$ norm respectively. Those norms are defined by:
\begin{align}
\|u\|_{L^r}:=
\begin{cases}
\left( \int_{\Omega} |u|^r \right)^{\frac{1}{r}}, &r<\infty,\\
\underset{x\in \Omega}{\mathrm{esssup}} |u(x)|, &r=\infty.
\end{cases}
\end{align} 
\begin{align}
\|u\|_{H_0^1}:=\|\nabla u\|_{L^2}.
\end{align}
The inner product for $L^2$ is notated by $\left<\cdot,\cdot \right>$ which is clearly defined by $\left<f,g 
\right>=\int fg$. The set $C^k$ in this paper is the space of $k$-times continuously differentiable functions. For any complex Banach space $\mathcal{X}$, and any $\Phi\subset \mathcal{X}$, the span of $\Phi$ is defined by 
\begin{align}
\label{span}
span_{\mathcal{X}}\Phi
:=
\left\{
\sum_{f\in \Phi} c_f f : 
c_f\in \mathbb{C}
\text{ for every }f\in \Phi  
\right\}.
\end{align}
Especially, when $\Phi=\{ \phi  \}$,
\begin{align}
span_{\mathcal{X}}\Phi= \{c\phi : c\in \mathbb{C}  \}.
\end{align}
When $\mathcal{X}$ is a Hilbert space with the inner product $\left<\cdot,\cdot \right>$, for any subspace $S\subset \mathcal{X}$, $S^\perp$ is defined by
\begin{align}
S^{\perp}:=\{ f\in \mathcal{X}: \left<f,g \right>=0 \text{ for every }g\in S  \}.
\end{align}
  Boldface letters will be used for elements in $\mathbb{R}^\mathbb{N}$ or $\mathbb{C}^{\mathbb{N}}$. For the convenience, for any constant $c\in \mathbb{C}$ and any $\bm{b}\in \mathbb{C}^{\mathbb{N}}$, notate $c\bm{b}:=(cb_1,cb_2,\dots)$. In the case then $\bm{c}\in \mathbb{C}^\mathbb{N}$, the notation $\bm{c}\bm{b}:=(c_1b_1,c_2b_2,\dots)$ is used. Especially, $\bm{y}=(y_1,y_2,\dots)\in U$ is used for stochastic parameter and $\bm{\nu}=(\nu_1,\nu_2,\dots)$ is used for multi-index. The size of the multi-index is defined by $|\bm{\nu}|=\sum_{i\geq 1} \nu$.  The class of multi-indices $\mathcal{F}=\{\bm{\nu}\in \mathbb{N}^\mathbb{N}: |\bm{\nu}|<\infty \}$ will be mainly used. The set of natural numbers, $\mathbb{N}$, contains $0$. For multi-index power or derivative, for any $\bm{\nu}\in \mathcal{F}$, for any analytic function $f:U\rightarrow \mathbb{C}$ and $\bm{y}\in U$, denote 
   \begin{align}
   \partial^{\bm{\nu}} f(\bm{y})=\frac{\partial^{|\bm{\nu}|}f\ }{\partial y_1^{\nu_1}\partial y_2^{\nu_2}\cdots }(\bm{y}),
   \end{align}
   and
\begin{align}
\bm{y}^{\bm{\nu}}=\prod_{i\geq 1}y_i^{\nu_i}.
\end{align}   
 As for multi-index combination, the following notation will be used: 
 \begin{align}
 \begin{pmatrix}
 \bm{\nu}\\\bm{m}
 \end{pmatrix}=\prod_{j\geq 1} 
 \begin{pmatrix}
 \nu_j\\m_j
 \end{pmatrix}.
 \end{align}
In the case when more than three multi-indices are involved, define
\begin{align}
\begin{pmatrix}
&\bm{\nu}\\
\bm{m}_{1} &\cdots &\bm{m}_p
\end{pmatrix}
:=
\prod_{j\geq 1}
\begin{pmatrix}
&\nu_j\\
m_{1j} &\cdots &m_{pj}
\end{pmatrix}:=\prod_{j\geq 1}
\frac{\nu_j!}{m_{1j}!m_{2j}!\cdots m_{pj}!}.
\end{align} 
Sometimes, an abused notation will be used for convenience. For example, for $J\subset \mathbb{N}$ and $\bm{y},\bm{z}\in \mathcal{F}$, it is defined:
\begin{align}
(\bm{z}_{J};\bm{y}_{-J})_j
:=
\begin{cases}
z_j, &j\in J,\\
y_j, &j\notin J.
\end{cases}
\end{align}
For a given $p\in (0,\infty)$, the space $\ell^p$ is defined by 
\begin{align}
\ell^p : = \left\{ \bm{b}\in \mathbb{C}^\mathbb{N}: \sum_{j=1}^{\infty} |b_j|^p <\infty   \right\}.
\end{align}
 Lastly, for a given positive integer $j$, $\bm{e}_j$ will be used as a multi-index whose entries are all zero except the $j$-th entry being $1$. Also, the notation $\partial_j^{n}:=\partial^{n\bm{e}_j}$ will be used when only a single variable is concerned.
\\

\section{Preliminary}
In this section, some necessary background for this paper will be provided, as well as some references. A special focus will be on $(\bm{b},\varepsilon)-$hololmorphy.

\subsection{$({b},\varepsilon)$- holomorphy} \label{behol} Throughout this paper, the functions of interest are the mappings from  $U:=[-1,1]^{\mathbb{N}}$ to $\mathcal{X}$ where $\mathcal{X}$ is a given Banach space chosen depending on the problems. Especially, a particular focus will be placed on the complex analytic or holomorphic extension of the maps. The domain of the holomorphic extension of interest in this paper is related to Berstein-ellipse.  For any $\rho>1$, the Berstein-ellipse is defined by:
\begin{align*}
\mathcal{E}_\rho
:=
\left\{
\frac{z+z^{-1}}{2}: z\in \mathbb{C}, 1\leq |z|<\rho
\right\}
\subset 
\mathbb{C}.
\end{align*} 
The ellipse defined above has the foci $\pm 1$, and the lengths of major and semi axes are $\rho+\frac{1}{\rho}$ and $\rho-\frac{1}{\rho}$ respectively. Especially, it contains the interval $[-1,1]$ for any $\rho>1$. To be specific, for any $r\in [-1,1]$, with elementary algebra, it is easy to see that the equation $\frac{z+z^{-1}}{2}=r$ always has a solution $z$ with $|z|=1$. For a given sequence of positive numbers $\bm{\rho}:=(\rho_j)_{j\geq 1}\subset (1,\infty)^\mathbb{N}$, define the poly-ellipse by
\begin{align}
\mathcal{E}_{\bm{\rho}}
:=
\prod_{j\geq 1}\mathcal{E}_{\rho_j}\subset \mathbb{C}^{\mathbb{N}}.
\end{align}
Throughout this paper, the topology on $\mathbb{C}^{\mathbb{N}}$ is the product topology, the weakest topology allowing every projection mapping to be continuous. For every $\bm{\rho}\in (1,\infty)^{\mathbb{N}}$, the topology on $\mathcal{E}_{\bm{\rho}}$ is the subspace topology of the product topology. One of the main tasks in this paper is to extend the domain of the ground eigenpairs defined on $U$ to the poly-ellipse.\\

\begin{definition} 
Let $\mathcal{X}$ be a given Banach space and let $O\subset \mathbb{C}^{\mathbb{N}}$ be a domain.
 A function $f: O\rightarrow  \mathcal{X}$  is called separately holomorpic(or separately complex analytic) if it is holomorphic for each coordinate separately fixing other coordinates.
 \\
\end{definition}

\begin{definition}
\label{holo}
Let $\mathcal{X}$ be a complex Banach space and let $\bm{b}=(b_j)_{j\geq 1}$ be a monotonically decreasing sequence of positive real numbers with $\bm{b}\in \ell^p$ for some $p\in (0,1]$. A map $u:U\rightarrow \mathcal{X}$  is called $(\bm{b},\varepsilon)$-holomorphy for some $\varepsilon>0$ if there exists a constant $M>0$ such that 
\begin{enumerate}
\item $u:U\rightarrow \mathcal{X}$ is continuous.

\item For every $(\bm{b},\varepsilon)$-admissible sequence $\bm{\rho}\in (1,\infty)^\mathbb{N}$, i.e.,
\begin{align}
\label{second}
\sum_{j\geq 1} b_j (\rho_j-1)\leq \varepsilon,
\end{align}
$u$ has a separately holomorphic extension on $\mathcal{E}_{\bm{\rho}}$.

\item For every $(\bm{b},\varepsilon)-$admissible $\bm{\rho}$, the following holds:
\begin{align}
\label{third}
\sup_{\bm{z}\in \mathcal{E}_{\bm{\rho}}}
\|u(\bm{z})\|_\mathcal{X}\leq M.
\end{align}

\end{enumerate}

\end{definition}

 In the definition of $(\bm{b},\varepsilon)-$holomorphy, the topology of $U$ is the subspace topology of the product topology on $\mathbb{C}^{\mathbb{N}}$. Thus, by Tychonoff's theorem, $U$ is compact. From the definition, for a given Banach space $\mathcal{X}$, in order to show that a map $u: U\rightarrow \mathcal{X}$ is $(\bm{b},\varepsilon)-$holomorphic, it is needed to find the decreasing sequence $\bm{b}\in \ell^p$ with constants $\varepsilon>0$ and $M>0$ that satisfies the three conditions in the definition. Some well explained literature about $(\bm{b},\varepsilon)-$holomorphy, can be found in \cite{zech, para3,para2}.

 \subsection{Parametric PDEs}
 \label{parapde} One of the goals of this paper is to show $(\bm{b},\varepsilon)-$holomorphy of the ground eigenpair of EVPs using the main result addressed in section 3. The parametric PDE problems depend on the parameters through the following form of affine parametric dependence:
 \begin{align}
 \label{paradep}
 L(\bm{y})=\phi_0+\sum_{j\geq 1} y_j \phi_j\in L^\infty,
 \end{align}
 where $\bm{y}:=(y_j)_{j\geq 1}\in U$ and $(\phi_j)_{j\geq 1}\subset L^\infty$. In general, parametric PDE problems can be represented by $\mathcal{P}(L(\bm{y}),u)=0$ where $u$ is the solution. Since $L(\bm{y})$ depends on the parameter $\bm{y}$, the solution has the dependence, so it can be notated by $u(\bm{y})$. In the case of EVP, the problem can be formulated by $\mathcal{P}(L(\bm{y}),u,\lambda)=0$, so that the eigenvalue $\lambda$ will also depend on $\bm{y}$.
 
  One of the main research streams in parametric PDEs is to have its coefficient functions depending on $L(\bm{y})$, and an incomplete list of such researches is \cite{alex2019, Bahn2024, alexey, Cohen2015, alexey2, para3}. For example, the following parametric Darcy flow has been considered many times:
  \begin{align}
  \label{parapdeex}
\begin{cases}
-\nabla\cdot(A(x,\bm{y})\nabla)u(x,\bm{y})=f(x), 
&(x,\bm{y})\in \Omega \times U,\\
u(x,y)=0, &(x,\bm{y})\in \partial\Omega \times U,
\end{cases}  
  \end{align}
  with 
  \begin{align}
  A(x,\bm{y})=\varphi_0 + \sum_{j\geq 1} y_j \varphi_j(x),
  \end{align}
  where $A(\bm{y})\in L^\infty$ and $(\varphi_j)_{j\geq 0}\subset L^\infty$. The dependence of the coefficient function $A$ on $\bm{y}$ is propagated to the dependence of the solution $u$ on $\bm{y}$, so the solution can be considered as a mapping $u: U\rightarrow H_0^1$. The affine dependence of $A$ on $\bm{y}$ is motivated by Karhunen-Lo\'eve (KL) expansion. By KL theorem, any square-integrable centered stochastic process has a KL expansion. Detailed explanation can be found in \cite{kar}.
  
   In addition, parametric PDEs having the boundary of the domain depending on $L(\bm{y})$ has been well studied. For example, instead of the coefficient function $A$, the domain $\Omega\subset \mathbb{R}^2$ in \eqref{parapdeex} can be parametrized by
\begin{align}
\Omega(\bm{y}):=
\left\{
(x_1,x_2):
0\leq x_1 \leq 1,\hspace{0.3 cm} 0\leq x_2\leq L(x_1,\bm{y})
\right\}.
\end{align}   
 An incomplete list of such studies can be found in \cite{jurgen, para3, jose1}.

\section{Main results} This section is devoted to providing a framework to show $(\bm{b},\varepsilon)$-holomorphy of the solutions to a class (possibly larger classes) of PDEs. The result from this section will be used to show that the ground eigenpairs of the EVPs stated in the previous sections satisfy $(\bm{b},\varepsilon)-$holomorphy in the next section. The main questions the following theorem tries to answer is that, when the bounds for derivatives are given, what is possible to say about $(\bm{b},\varepsilon)$-holomorphy of the function, what additional conditions are needed, and then what would be the possible $\bm{b}$ and $\varepsilon$.  Now, the main theorem is presented below:\\

\begin{theorem}
\label{main}
Let $(\mathcal{X},\|\cdot\|_{\mathcal{X}})$ be a given Banach space and let $u:U\rightarrow \mathcal{X}$ be a continuous function. Suppose that $u$ is separately holomorphic extension. In other words, $u_j(z_j):=u(z_j;\bm{y}_{-j})$ is holomorphic(complex analytic) on  some open set $V_j\supset [-1,1]$ for all $j\geq 1$ and that, for every $j\geq 1$, the following bound holds: for any $n\geq 1$,
\begin{align}
\label{deri-bound}
 \sup_{\bm{y}\in U}\|\partial_j^n u(\bm{y})\|_{\mathcal{X}}\leq \alpha_n \beta^n_j,
\end{align} 
where $\bm{\beta}:=(\beta_j)_{j\geq 1}$ is a sequence of positive numbers and $\bm{\alpha}:=(\alpha_j)_{j\geq 1}$ is a sequence of non-negative numbers satisfying
\begin{align}
\label{radius}
0<\varepsilon=\lim_{n\rightarrow \infty}\frac{\alpha_n(n+1)}{\alpha_{n+1}}<\infty,
\end{align}
 for some constant $\varepsilon>0$. Further suppose that there exists $\bm{\gamma}\in (1,\infty)^{\mathbb{N}}$ such that $u$ has a jointly continuous extension onto $H_{\bm{\gamma}}$ where
  \begin{align}
 \label{H}
 H_{\bm{\gamma}} : =\prod_{j\geq 1} \left\{ z\in \mathbb{C} : \min_{y\in [-1,1]} |z-y|<\frac{\varepsilon}{\gamma_j \beta_j} \right\}.
 \end{align}
Then, if $\bm{b}:=\bm{\gamma \beta }\in \ell^p$ with $p\in (0,1]$ is a decreasing,  $u:U\rightarrow \mathcal{X}$ satisfies $(\bm{b},\varepsilon)-$holomorphy.
\end{theorem}

\begin{proof}
 Let an integer $j\geq 1$ and a parameter $\bm{y}\in U$ be given. Let $\bm{z}=(z_j;\bm{y}_{-j})$ where $z_j\in \mathbb{C}$. Due to the separately holomorphic extension condition on $\prod_{j\geq 1} V_j$, the Taylor expansion of $u$ at $\bm{y}\in U$ in $z_j$ can be written as
\begin{align}
u(\bm{z})=\sum_{n=0}^{\infty}\partial_j^nu(\bm{y})\frac{(z_j-y_j)^n}{n!}.
\end{align}
Then, by the given bound \eqref{deri-bound} for the derivatives of $u$ with the triangular inequality, observe that
\begin{align}
\label{holo1}
\|u(\bm{z})\|_{\mathcal{X}}
\leq
\sum_{n=0}^{\infty}
\left\|
\partial_j^nu(\bm{y})
\right\|_{\mathcal{X}}
\frac{|z_j-y_j|^n}{n!}
\leq
 \sum_{n=0}^{\infty}
 \frac{\alpha_n}{n!}
 \beta_j^n
 |z_j-y_j|^n.
\end{align}
Now, note that the radius of the convergence of this series is
\begin{align}
\label{radi}
\lim_{n\rightarrow \infty}
\frac
{\frac{\alpha_n}{n!}
 \beta_j^n}
{\frac{\alpha_{n+1}}{(n+1)!}
 \beta_j^{n+1}}
 =
 \lim_{n\rightarrow \infty}
 \frac{\alpha_n(n+1)}{\alpha_{n+1}\beta_j}=\frac{\varepsilon}{\beta_j}.
\end{align}
 Therefore, $u_j(z_j):=u(z_j;\bm{y}_{-j})$ is holomorphic in the open ball of radius ${\varepsilon}/{\beta_j}$  centered at $y_j$. 
 Note that, at the boundary of the circle of radius $\varepsilon/\beta_j$, the convergence is undetermined. Thus, for now, the holomorphy is only defined in the open ball. In order to use the compactness of the domain later, the domain of interest will be restricted to the closed ball of radius ${\varepsilon}/{(\gamma_j \beta_j)}$ for $\gamma_j >1$ which is given in the statement. So far, it has been proven that $u$ is separately holomorphic on
 \begin{align}
 \label{H}
 H_{\bm{\gamma}} : =\prod_{j\geq 1} H_{\gamma_j},
 \end{align}
 where, for every $j\geq 1$,
 \begin{align}
 H_{\gamma_j}:= \bigcup_{y_j\in [-1,1] }B\left(y_j,\varepsilon/\gamma \beta_j\right)=\left\{ z\in \mathbb{C} : \min_{y\in [-1,1]} |z-y|\leq \frac{\varepsilon}{\gamma_j \beta_j} \right\},
 \end{align}
 where $B(c,r)=\{z\in \mathbb{C}: |z-c|\leq  r\}$ is a closed ball. 
 
  Now let $\bm{\rho}=(\rho_j)_{j\geq 1}$ be a given $(\bm{\gamma}\bm{\beta}, \varepsilon)-$admissible sequence assuming that $\bm{b}:=\bm{\gamma \beta}\in \ell^p$ for some $p\in (0,1]$ is a decreasing sequence. In other words, $\bm{\rho}$ satisfies the following inequality:
\begin{align}
\label{main_admissible}
\sum_{j=1}^\infty \gamma_j \beta_j(\rho_j-1)\leq \varepsilon.
\end{align} 
This condition implies that  $\rho_j<1+\frac{\varepsilon}{\gamma_j \beta_j}$ with some manipulation from $\gamma_j \beta_j(\rho_j-1)\leq \varepsilon$ for every $j\geq 1$. Recall that the Bernstein ellipse $\mathcal{E}_{\rho_j}$ has foci $\pm 1$ and the length of the minor-axis $\frac{1}{2}\left(\rho_j-\frac{1}{\rho_j}\right)$. Note that this length is a strictly increasing function in $\rho_j$. To be specific, with the function defined by 
\begin{align}
\mathcal{R}_{minor}(\rho) : =\frac{1}{2}\left( \rho -\frac{1}{\rho}  \right),
\end{align}
observe that
\begin{align}
\label{miinc}
\frac{d}{d\rho} \mathcal{R}_{minor}(\rho) = \frac{1}{2}\left(1+\frac{1}{\rho^2} \right)>0,
\end{align}
for all $\rho >1$.
 Therefore, it holds that
\begin{align}
\label{semiaxis}
\frac{1}{2}
\left(
\rho_j-\frac{1}{\rho_j}
\right)
&< 
\frac{1}{2}
\left(
\frac{\gamma_j \beta_j+\varepsilon}{\gamma_j  \beta_j}-\frac{\gamma_j  \beta_j}{\gamma_j  \beta_j+\varepsilon}
\right)
\notag\\
&=
\frac{1}{2}\frac{\varepsilon^2+2\gamma_j  \beta_j \varepsilon}{\gamma_j  \beta_j(\gamma_j  \beta_j+\varepsilon)}
\notag\\
&=
\frac{\varepsilon}{\gamma_j  \beta_j}\left(\frac{\varepsilon+2\gamma_j  \beta_j}{2\gamma_j  \beta_j+2\varepsilon} \right)
\notag\\
&<
 \frac{\varepsilon}{\gamma_j  \beta_j},
\end{align}
where the first inequality is by \eqref{miinc}.
Note that the the length of the major-axis $\frac{1}{2}\left(\rho_j+\frac{1}{\rho_j}\right)$ is also an increasing function in $\rho_j$ because $\rho_j>1$. To be specific, with the function defined by 
\begin{align}
\mathcal{R}_{major}(\rho) : =\frac{1}{2}\left( \rho +\frac{1}{\rho}  \right), 
\end{align}
observe that
\begin{align}
\label{mainc}
\frac{d}{d\rho} \mathcal{R}_{major}(\rho) = \frac{1}{2}\left(1-\frac{1}{\rho^2} \right)>0,
\end{align}
for all $\rho >1$. Thus, the distance between the right end of the ellipse in the direction of the major axis and the foci $+1$ is
\begin{align}
\label{majoraxis}
\frac{1}{2}
\left(
\rho_j+\frac{1}{\rho_j}
\right)-1
&<
\frac{1}{2}
\left(
\frac{\gamma_j  \beta_j+\varepsilon}{\gamma_j  \beta_j}+\frac{\gamma_j  \beta_j}{\gamma_j  \beta_j+\varepsilon}
\right)-1
\notag\\
&=
\frac{1}{2}\frac{\varepsilon^2+2\gamma_j  \beta_j \varepsilon+2\gamma_j ^2 \beta_j^2}{\gamma_j  \beta_j(\gamma_j  \beta_j+\varepsilon)}-1
\notag\\
&<
\frac{1}{2}
\left(
\frac{2(\gamma_j  \beta_j+\varepsilon)^2}{\gamma_j  \beta_j(\gamma_j  \beta_j+\varepsilon)}
\right)-1
\notag\\
&=
\left(\frac{\varepsilon+\gamma_j  \beta_j}{\gamma_j  \beta_j} \right)-1
\notag\\
&
=
 \frac{\varepsilon}{\gamma_j  \beta_j},
\end{align}
where the first inequality is by \eqref{mainc}.

So far, it has been shown that, for each $y_j\in [-1,1]$, $u$ has a separately holomorphic extension whose domain contains the closed ball in $\mathbb{C}$ centered at $y_j$ for every  $y_j\in [-1,1]$ with radius ${\varepsilon}/{(\gamma \beta_j)}$ in \eqref{radi}. 
 Then \eqref{semiaxis} and \eqref{majoraxis} imply that, because the radius of the domain of separately holomorphic extension is larger than the length of the semi-axis of $\mathcal{E}_{\rho_j}$ and larger than the distance between foci and the vertex, the Bernstein ellipse $\mathcal{E}_{\rho_j}$ is contained in the region of the separately holomorphy of $u$.
In other words, to be specific, the following inclusion holds:
\begin{align}
\label{ext}
\mathcal{E}_{\rho_j} \subset H_{\gamma_j}= \bigcup_{y_j\in [-1,1] }B\left(y_j,\varepsilon/\gamma_j \beta_j\right).
\end{align} 

   Since $u_j$ has a holomorphic extension on $H_{\gamma_j}$ and $\mathcal{E}_{\rho_j}\subset H_{\gamma_j}$ for all $j\geq 1$ is verified from \eqref{ext}, $u$ has separately holomorphic extension on $\mathcal{E}_{\bm{\rho}}$. It proves the second condition of $(\bm{\gamma}\bm{\beta},\varepsilon)-$holomorphy, for example, in \eqref{second}. 

 Lastly, it is left to show that the last condition for $(\bm{\gamma} \bm{\beta},\varepsilon)$-holomorphy, the existence of the uniform bound $M$ in \eqref{third}. It will be shown by using the continuity of $u$ and the compactness of $H_{\bm{\gamma}}$. Recall that, from the statement, the sequence $\bm{\gamma}$ is already chosen to have $u$ continuous on $H_{\bm{\gamma}}$. Thus, due to the continuity of $u: H_{\bm{\gamma}}\rightarrow \mathcal{X}$, it is only necessary to show the compactness of $H_{\bm{\gamma}}$.
 The compactness can be easily seen by Tychonoff's theorem because $H_{\bm{\gamma}}$ is the infinite product of compact sets $H_{\gamma_j}$'s. (An elementary proof of this is provided in Remark \ref{prodtop}.) 
 Therefore, since the image of a compact set under a continuous function is compact, $\|u(H_{\bm{\gamma}})\|_{\mathcal{X}}:=\{\|u(\bm{z})\|_{\mathcal{X}}:\bm{z}\in H_{\bm{\gamma}}   \}\subset \mathbb{R}$ is compact. Then, again by the Heine-Borel theorem, $\|u(H_{\bm{\gamma}})\|_{\mathcal{X}}$ is bounded.  Thus, it concludes that there exists a constant $M_{\bm{\gamma}}$ defined by 
\begin{align}
\label{Mbound}
M_{\bm{\gamma}}:=\sup_{\bm{z}\in H_{\bm{\gamma}}}\|u(\bm{z})\|_{\mathcal{X}}\geq \sup_{\bm{z}\in \mathcal{E}_{\bm{\rho}}}\|u(\bm{z})\|_{\mathcal{X}},
\end{align}
  for every $(\bm{\gamma} \bm{\beta},\varepsilon)-$admissible sequence  $\bm{\rho}$ by the result of \eqref{ext}. Thus, by with the constant $M_{\bm{\gamma}}$,  it concludes that \eqref{third} is satisfied. Combining all the results above, it can be concluded that $u$ is $(\bm{b},\varepsilon)-$holomorphic. 
\end{proof}

\begin{remark}

 Note that, in the proof, the condition for $\bm{\rho}$ in \eqref{main_admissible} is not fully used. In fact, only the weaker condition
 \begin{align}
 \max_{j\geq 1} \gamma_j \beta_j (\rho_j-1) \leq \varepsilon
 \end{align}
 is needed for the argument. As this weaker condition only allows to have larger class of $\bm{\rho}$. It only  makes the argument working for larger class of $\bm{\rho}$ than the class of $(\bm{\gamma\beta},\varepsilon)-$admissible that we actually need. It is still open if, with the exact condition \eqref{main_admissible}, more general class of functions $u$ than the one in Theorem \ref{main} could be shown to be $(\bm{b},\varepsilon)-$holomorphic.  

\end{remark}

\begin{remark}
From the result of Theorem \ref{main}, one might think that, since $u$ is separately holomorphic on $H_{\bm{\gamma}}$, by simple application of Cauchy's integral formula, 
\begin{align}
\sup_{\bm{y}\in U} \| \partial_j^n u(\bm{y}) \|_\mathcal{X}
\leq 
\alpha_n (\gamma_j \beta_j)^n
\end{align}
where 
\begin{align}
\alpha_n = \frac{n!}{\varepsilon^n} \sup_{\bm{Z}\in H_{\bm{\gamma}}}\|u(\bm{z})\|_{\mathcal{X}},
\end{align}
which seems recovering the condition similar with \eqref{deri-bound} that is assumed in Theorem \ref{main}. This could look as if the condition \eqref{deri-bound} unnecessary. However, note that, at the beginning of the proof, we did not know if $u$ is separately holomorphic on $H_{\bm{\gamma}}$. Indeed, the bound \eqref{deri-bound} has been used to show that $u$ is separately holomorphic on $H_{\bm{\gamma}}$.

\end{remark}

\begin{remark}
\label{prodtop}
 Note that, although $H_{\bm{\gamma}}$ is the infinite product of the compact sets $H_{\gamma_j}$, the size of $H_{\gamma_j}$ can diverge to infinity as $j$ grows. Thus, one can have doubts about the compactness of $H_{\bm{\gamma}}$. For such doubt, it is not difficult to present an elementary proof of the compactness. Indeed, recall that an open set in the product topology of $\mathbb{C}^\mathbb{N}$ is an arbitrary union of the sets of the form $\prod_{j\geq 1}O_j$ where, for each $j\geq 1$, $O_j\subset\mathbb{C}$ is an open set, but $O_j\neq \mathbb{C}$ only for finitely many $j\geq 1$. In other words, the basis $\mathcal{B}$ of the product topology is the collection of such open sets. Now, suppose that $\mathcal{C}:=\{ O^\alpha \}_{\alpha\in A}$ is an arbitrary open cover of $H_{\bm{\gamma}}$ where $A$ is an index set. With this, the existence of the finite open subcover will be shown. From an elementary fact from topology, it is enough to assume that  $\mathcal{C}\subset \mathcal{B}$. Then note that, because an open set $O\in \mathcal{C}$ already covers infinitely many dimensions of $H_{\bm{\gamma}}$ with $\mathbb{C}$, with the specified $O$, it suffices to cover the rest of the dimensions which is now finite. Let $O=\prod_{j\geq 1} O_j$ with only finitely many $O_j$ is open and not equal to $\mathbb{C}$. Let $O^\alpha=\prod_{j\geq 1}O_j^\alpha$ for each $\alpha\in A$ and let $I=\{ j\geq 1: O_j\neq \mathbb{C}  \}$ which is a finite set. Then, note that, for any $j\in I$, since $H_{\gamma_j}$ is compact (by Heine-Borel theorem since closed and bounded), there is a finite subset $S_j\subset A$ such that $\{O_j^\alpha\}_{\alpha\in S_j}$ covers $H_{\gamma_j}$. Since $S_j$ is finite for each $j$ and there are only finitely many $j\in I$, letting $S:=\bigcup_{j\in I}S_j$, the family $\{O\}\cup\{O^\alpha\}_{\alpha\in S}$ is a finite cover $H_{\bm{\gamma}}$. This proves that $H_{\bm{\gamma}}$ is compact.
 \end{remark}

Theorem \ref{main} suggests a general framework to check $(\bm{ b},\varepsilon)-$holomorphy, with possible $\bm{ b}\in \ell^p$ and $\varepsilon>0$, of the solution to some classes of PDEs. First of all, note that, due to Theorem \ref{main}, it is possible to check the holomorphy without dealing with infinite-dimensional complex analysis and multi-index notations. The following strategy, utilizing Theorem \ref{main}, will be used to show the $(\bm{ b},\varepsilon)-$holomorphy of the solution to a given parametric PDE:
 \\
 \begin{enumerate}

 \item Separately holomorphic extension: the Banach space version of the implicit function theorem will be used to show the holomorphic (complex analytic) dependence of the solution on each parameter. 
 
 \item Bound on the single parametric derivatives: the real inductive argument used in \cite{alexey, alex2019, Bahn2024} will be used to obtain the bound of the form of \eqref{deri-bound}, but with a simpler argument, especially without dealing with multi-index.
 
 \item  Continuous dependence on parameters: Banach space version of the implicit function theorem will be used to show the continuous dependence of the solution $u$ on the coefficient functions. Then, $\bm{\gamma}\in (1,\infty)^{\mathbb{N}}$ is chosen so that the image of the parameter space $H_{\bm{\gamma}}$ under the coefficient function is in the domain of continuity for $u$.

 \end{enumerate}

 With the same condition as Theorem \ref{main}, it is possible to obtain the bounds for mixed derivatives in the following theorem. It is notable that the bound can be recovered even if the multi-index arguments are avoided, thanks to Cauchy's integral formula. Furthermore, the recovered bounds are enough to guarantee quasi-Monte Carlo methods. 
\\

\begin{theorem}
\label{impb}
Suppose all the notations and conditions given in Theorem \ref{main} hold. Then the following bound holds:
\begin{align}
\label{imp-bound}
\left\|  \partial^{\bm{\nu}}u(\bm{y}) \right\|_{\mathcal{X}} \leq  M_{\bm{\gamma}} \bm{\nu}! \left(\frac{\bm{\gamma}\bm{\beta}}{\varepsilon}\right)^{\bm{\nu}},
\end{align}
for every $\bm{\nu}\in \mathcal{F}$ and every $\bm{y}\in U$. The constant $M_{\bm{\gamma}} >0$ is from \eqref{Mbound}.
\end{theorem}

\begin{proof}
Let $\bm{\nu}\in \mathcal{F}$ and $\bm{y}\in U$ be given. Recall that, from the proof of Theorem \ref{main}, the map $u: U\rightarrow \mathcal{X}$ has a separately holomorphic extension to $H_{\bm{\gamma}}=\prod_{j\geq 1} H_{\gamma_j}$ defined in \eqref{H}. 
Define
\begin{align}
S_{\bm{\nu}}:=
\left\{
j\geq 1 : 
\nu_j\neq 0
\right\}.
\end{align}
Then, by Hartog's theorem, since $|S_{\bm{\nu}}|<\infty$, $u$ is holomorphic in $H^{S_{\bm{\nu}}}_{\bm{\gamma}}$ with 
\begin{align}
H^{S_{\bm{\nu}}}_{\bm{\gamma}} : = \prod_{j\in S_{\bm{\nu}}}H_{\gamma_j}.
\end{align}
Then the Cauchy integral formula can be applied to represent $\partial^{\bm{\nu}} u( \bm{y})$ by 
\begin{align}
\label{cauchy}
\partial^{\bm{\nu}} u(\bm{y})
=
\frac{\bm{\nu}!}{(2\pi i )^{|S_{\bm{\nu}}|}}
\oint_{\mathcal{C}(\bm{y},\bm{\nu})}
\frac
{u(\bm{z}_{S_{\bm{\nu}}};\bm{y}_{-S_{\bm{\nu}}})}
{
\prod_{j\in S_{\bm{\nu}}}
(z_j-y_j)^{\nu_j+1}
}
\prod_{j\in S_{\bm{\nu}}} dz_j,
\end{align}
where 
\begin{align}
\mathcal{C}(\bm{y},\bm{\nu})
:=
\prod_{j\in S_{\bm{\nu}}}
\mathcal{C}_j(\bm{y}),
\end{align}
and
\begin{align}
\mathcal{C}_j(\bm{y})
:=
\left\{
z\in \mathbb{C}:
|z-y_j|=\frac{\varepsilon}{\gamma_j \beta_j}
\right\}.
\end{align} 
Then observe that,
\begin{align}
\|\partial^{\bm{\nu}} u(\bm{y})\|_{\mathcal{X}}
&
\leq
\frac{\bm{\nu}!}{|2\pi i|^{|S_{\bm{\nu}}|}}
\oint_{\mathcal{C}(\bm{y},\bm{\nu})}
\frac
{\left\| u(\bm{z}_{S_{\bm{\nu}}};\bm{y}_{-S_{\bm{\nu}}})\right\|_{\mathcal{X}}}
{
\prod_{j\in S_{\bm{\nu}}}
|z_j-y_j|^{\nu_j+1}
}
\prod_{j\in S_{\bm{\nu}}} |dz_j|
\notag\\
&
\leq
\frac{\bm{\nu}!}{|2\pi i|^{|S_{\bm{\nu}}|}}
\oint_{\mathcal{C}(\bm{y},\bm{\nu})}
\frac
{M_{\bm{\gamma}} }
{
\frac{\varepsilon^{|\bm{\nu}|}}{\bm{\gamma}^{\bm{\nu}}\bm{\beta}^{\bm{\nu}}}
\left(\prod_{j\in S_{\bm{\nu}}}
\frac{\varepsilon}{\gamma_j  \beta_j}\right)
}
\prod_{j\in S_{\bm{\nu}}} |dz_j|
\notag\\
&
=
\frac{\bm{\nu}! \bm{\beta}^{\bm{\nu}}}{|2\pi i|^{|S_{\bm{\nu}}|}}
\frac{M_{\bm{\gamma}}  \bm{\gamma}^{\bm{\nu}} }{\varepsilon^{|\bm{\nu}|}}\left(\prod_{j\in S_{\bm{\nu}}} \frac{\gamma_j  \beta_j}{\varepsilon} \right)
\oint_{\mathcal{C}(\bm{y},\bm{\nu})}
\prod_{j\in S_{\bm{\nu}}} |dz_j|
\notag\\
&
=
\frac{\bm{\nu}! \bm{\beta}^{\bm{\nu}}}{|2\pi i|^{|S_{\bm{\nu}}|}}
\frac{M_{\bm{\gamma}}  \bm{\gamma}^{\bm{\nu}}}{\varepsilon^{|\bm{\nu}|}}\left(\prod_{j\in S_{\bm{\nu}}} \frac{\gamma_j  \beta_j}{\varepsilon} \right)
\left( \prod_{j\in S_{\bm{\nu}}}2\pi\frac{\varepsilon}{\gamma_j  \beta_j} \right)
\notag\\
&
=
\frac{ M_{\bm{\gamma}} \bm{\gamma}^{\bm{\nu}}}{\varepsilon^{|\bm{\nu}|}}
\bm{\nu}! \bm{\beta}^{\bm{\nu}},
\end{align}
where the first inequality is simple triangular inequality, the second inequality is by the definition of $M$ and $\mathcal{C}(\bm{y},\bm{\nu})$, and the second last equality is by the definition of $\mathcal{C}(\bm{y},\bm{\nu})$ again.  Therefore, the desired result \eqref{imp-bound} is obtained. 
\end{proof}

\begin{remark}
 It is worth noting that a function $u:U\rightarrow \mathcal{X}$ is $(\bm{b},\varepsilon)-$holomorphic if and only if it is $(\bm{b}/\varepsilon, 1)-$holomorphic. This equivalence means $\bm{b}$ and $\varepsilon$ take the same role. It is clear to see that the role of $\varepsilon$ together with $\bm{b}$ in Definition \ref{holo} is controlling the size of the region of holomorphic extension by controlling $\bm{\rho}$. Since $\bm{b}$ is a sequence of positive decreasing real numbers, $\bm{\rho}$ could be an increasing sequence diverging to infinity. Larger $\varepsilon$ or smaller $\bm{b}$ allows the function $u$ to have a broader holomorphic extension. In the case of Theorem \ref{main}, $\bm{b}$ and $\varepsilon$ are determined by the prior upper bound \eqref{deri-bound}. From this perspective, it is clear that better or tighter upper bound \eqref{deri-bound} gives a better bound \eqref{imp-bound}. For example, let $\|\partial_j^n u(\bm{y})\|_{\mathcal{X}}\leq Cn!(kb_j)^n$ for some constant $C>0$ and $k>0$. Then, in the spirit of \eqref{radi}, the radius of convergence of the Taylor expansion will be
\begin{align}
\label{exp-radi}
\lim_{n\rightarrow \infty} \frac{\frac{Cn!(kb_j)^n}{n!}}{\frac{C(n+1)!(kb_j)^{n+1}}{(n+1)!}  }
=
\frac{1}{kb_j}.
\end{align}
In this case, $\varepsilon=\frac{1}{k}. $ From \eqref{exp-radi}, if it happens to have a worse bound (with larger $k$ or larger $b_j$), the radius of convergence gets smaller, which shows the clear connection between the role of $\varepsilon$ in Definition \ref{holo} and $\varepsilon$ in \eqref{deri-bound}.
\end{remark}

\begin{remark}

In the case when $H_{\bm{\gamma}} \subset \mathcal{V}$ with $\bm{\gamma} = (\gamma,\gamma,\dots)$ for $\gamma>1$, the bound in \eqref{imp-bound} is now
\begin{align}
\label{imp-bound12}
\left\|  \partial^{\bm{\nu}}u(\bm{y}) \right\|_{\mathcal{X}} \leq  M_{\bm{\gamma}} \bm{\nu}! \left(\frac{{\gamma}\bm{\beta}}{\varepsilon}\right)^{\bm{\nu}},
\end{align}
which is a best possible bound for a given $\bm{\beta}$ based on the argument in this paper.

\end{remark}

\section{Parametric EVPs}
In this section, two PDEs of special interest, linear and semilinear EVPs, are studied. In the following EVPs, a part of coefficient functions depends on parameters $\bm{y}\in U$ as in Section \ref{parapde}, resulting in the parametric dependence of the ground eigenpairs. The EVPs of this paper concern the case of solution space being a complex Banach space $H_0^1$. In other words, the ground state $u$ and the ground eigenvalue $\lambda$ can be thought of as mappings from $U$ to $H_0^1$ and from $U$ to $\mathbb{C}$ respectively. The goal of this section is to show that the ground eigenpairs of the following EVPs are $(\bm{b},\varepsilon)-$holomorphic.

\subsection{Linear Eigenvalue Problem}
\label{linear}
 The focus of this subsection is on showing that parametric linear EVP \eqref{problem1} has $(\bm{b},\varepsilon)-$holomorphic ground eigenpair.
The problem is given as follows:
\begin{align}
\label{problem1}
\begin{cases}
-\nabla\cdot(A(x,\bm{y})\nabla)u(x,\bm{y})
+B(x,\bm{y})u(\bm{y})=C(x,\bm{y})\lambda(\bm{y})u(x,\bm{y}),&(x,\bm{y})\in \Omega\times U,
\\
u(x,\bm{y})=0,&(x,\bm{y})\in \partial\Omega\times U,
\end{cases}
\end{align}
where $U=[-1,1]^{\mathbb{N}}$ is the space of parameters, and $\Omega\subset\mathbb{R}^d$ is bounded convex domain with Lipschitz boundary. In order to have non-zero solutions, the normalization condition, $\|u(\bm{y})\|_{L^2}=1$ for all $\bm{y}\in U$, is enforced. As for the coefficient functions $A(x,\bm{y})$, $B(x,\bm{y})$ and $C(x,\bm{y})$, the following conditions will be assumed throughout this section:
\\

\begin{assumption}
\label{assumption1}
\begin{itemize}
\item The coefficient functions have the following representations:
\begin{align}
A(x,\bm{y})&=A_0(x)+\sum_{j\geq 1}y_j A_j(x),\\
B(x,\bm{y})&=B_0(x)+\sum_{j\geq 1} y_j B_j(x),\\
C(x,\bm{y})&=C_0(x)+\sum_{j\geq 1} y_j C_j(x),
\end{align}
with  $(A_j)_{j\geq 0}, (B_j)_{j\geq 0},(C_j)_{j\geq 0} \subset L^\infty$.
\\

\item There exist constants $\overline{A},\overline{B},$ and $\overline{C}$ such that
\begin{align}
\begin{matrix}
\left\|A(x,\bm{y}) \right\|_{L^\infty}
\leq \overline{A},
&\left\|B(x,\bm{y}) \right\|_{L^\infty}
\leq \overline{B},
&\left\|C(x,\bm{y}) \right\|_{L^\infty}
\leq \overline{C}.
\end{matrix}
\end{align}

\item There exist constants $\underline{A}>0$ and $\underline{C}>0$ such that 
\begin{align}
\label{underline}
\underline{A}\leq A(x,\bm{y})\text{, and  } \underline{C}\leq C(x,\bm{y}),
\end{align}
for all $(x,\bm{y})\in \Omega\times U$.

\item Let $\bm{\beta}:=  \zeta \bm{c} $ be defined with $ c_j:= \max\{\|A_j\|_{L^\infty}, \| B_j  \|_{L^\infty}, \| C_j  \|_{L^\infty}  \}$ for some constant $\zeta>1$ that will be specified in Lemma \ref{LEVP_bound}. There exists $\bm{\gamma}\in (1,\infty)^\mathbb{N}$ such that $\bm{\gamma \beta}\in \ell^p$ for some $p\in (0,1]$ is a sequence of decreasing positive numbers and 
\begin{align}
\label{Gamm}
\Gamma:=\sum_{j\geq 1}\frac{1}{\gamma_j}< \min\left\{ 1, 4{\underline{D}}  \right\},
\end{align}
where $\underline{D}:=\min\{ \underline{A}, \underline{C} \}$.

\end{itemize}
\end{assumption}

This problem has been well studied, and a short list of the references is \cite{andreev2010,alex2019, alexey}. In \cite{andreev2010}, it is assumed that  $B\equiv 0$ and $C\equiv 1$, and, in \cite{alex2019}, it is assumed that the function $C$ has no parametric dependence. The problem \eqref{problem1} is the same as the one in \cite{alexey} except that only one special kind of parametric dependence is considered in this paper. Also, it will follow a similar analysis for the part of bounding the derivatives. The function of interest is the ground eigenpair. The holomorphic dependence of the ground state on the parameters near $U$ is shown in Corollary 2 in \cite{andreev2010} when $B\equiv 0$ and $C\equiv 1$. Recently, in \cite{alex2019}, the upper bound for the mixed derivative is firstly obtained, $\|\partial^{\bm{\nu}}u(\bm{y})\|_{H_0^1}\leq R(|\bm{\nu}|!)^{1+\varepsilon}\bm{\beta}^{\bm{\nu}}$ for any $\varepsilon>0$, for any $\bm{\nu}\in \mathcal{F}$, and some $R>0$.  In \cite{alexey}, with new analysis, the better upper bound $\|\partial^{\bm{\nu}}u(\bm{y})\|_{H_0^1}\leq R(|\bm{\nu}|!)\bm{\beta}^{\bm{\nu}}$ is obtained. In this section, as corollaries of Theorem \ref{main} and Theorem \ref{impb}, $(\bm{b},\varepsilon)$-holomorphy and recovering the mixed derivative bound will be shown.
\\

 \begin{corollary}
 \label{LEP}
 Let $(u(\bm{y}),\lambda(\bm{y}))$ is the ground eigenpair of \eqref{problem1} under Assumption \ref{assumption1}. 
  Then $u:U\rightarrow H_0^1$ and $\lambda:U\rightarrow \mathbb{R}$ is $(\bm{b},\varepsilon)-$holomorphic with $\bm{b}:=\bm{\gamma \beta}$ and  $\varepsilon=\frac{1}{4}$. 
 \end{corollary}

The proof of this corollary will be the combination of the lemmata from the following subsections.

\subsubsection{Separately holomorphy of the ground eigenpairs}
\label{linseparate}

In this subsection, the separately holomorphic property of the ground eigenpairs is shown. In the proof, the implicit function theorem will be used.

\begin{lemma}
\label{LEVP_separate}
The ground eigenpair $(u(\bm{y}),\lambda(\bm{y})):U\rightarrow H_0^1\times \mathbb{R}$ of \eqref{problem1} has separately holomorphic extension onto $\mathcal{O}$ where $U\subset \mathcal{O}\subset \mathbb{C}^\mathbb{N}$ with $\mathcal{O}:=\prod_{j\geq 1}\mathcal{O}_j $ where  $\mathcal{O}_j\subset \mathbb{C}$ is an open set for all $j\geq 1$.
\end{lemma}
 
 \begin{proof}
  A similar method with the proof of Lemma \ref{LEVP1} will be used to show the existence of the separately analytic extension of the ground eigenpair. Let $\bm{y}\in U$ and $j\geq 1$ be given. Then let $(u(\bm{y}),\lambda(\bm{y}))$ is the solution to \eqref{problem1}.  First, define a mapping on $\mathbb{C}^2\times H_0^1$ by
  \begin{align}
  \mathcal{N}(\xi,\mu,w):=
  \left(
  -\left(
  \nabla\cdot(A(\bm{y}+\xi \bm{e}_j)\nabla 
  \right)
  +B(\bm{y}+\xi \bm{e}_j)
  -
  (\lambda(\bm{y})+\mu)
  \right)C(\bm{y}+\xi \bm{e}_j)
  (u(\bm{y})+w).
  \end{align}
Note that $\mathcal{N}(0,0,0)=0$. With very similar computation with \eqref{Gderi}, the Gateaux derivative of $\mathcal{N}$ is obtained as follows:
\begin{align}
D_{(\mu,w)}\mathcal{N}(0,0,0)(\mu,w)
&=
\lim_{h\rightarrow 0}
\frac{N(0,h\mu,h w)-\mathcal{N}(0,0,0)}{h}
\notag\\
&=
(-\nabla\cdot(A(\bm{y})\nabla)+B(\bm{y})-C(\bm{y}) \lambda(\bm{y}))w-\mu C(\bm{y}) u(\bm{y})
.
\end{align}
 With the same argument with the proof of Lemma \ref{LEVP1},  $D_{\mu,w}\mathcal{N}(0,0,0):\mathbb{C}\times E\rightarrow H^{-1} $ is a bijection. Lastly, with similar computation with \eqref{inversebound}, there is a constant $R>0$ such that
\begin{align}
\| (D_{(\mu,w)}\mathcal{N}(0,0,0))^{-1}\psi  \|_{\mathbb{C}\times H_0^1}^2
&\leq R\|\psi\|_{H^{-1}}^2,
\end{align}
for any $\psi\in H^{-1}$.
 It shows that $(D_{(\mu,w)}\mathcal{N}(0,0,0))^{-1}$ is bounded so $(D_{(\mu,w)}\mathcal{N}(0,0,0))$ is an isomorphism. Therefore, $\mathcal{N}$ is Fr\'echet differentiable with the continuous inverse, and so, by the implicit function theorem, there exists $\mathcal{O}_j\subset \mathbb{C}$ such that there exist complex analytic functions $\mu:\mathcal{O}_j\rightarrow \mathbb{C}$ and $w:\mathcal{O}_j\rightarrow H_0^1$ of $\xi$, and $\mathcal{N}(\xi, \mu(\xi),w(\xi))=0$ for all $\xi\in \mathcal{O}_j$. In other words, $u(\bm{y}+\xi \bm{e}_j):=u(\bm{y})+w(\xi)$ and $\lambda(\bm{y}+\xi \bm{e}_j):=\lambda(\bm{y})+\mu(\xi)$ are the complex analytic function in $\xi$ and are the ground eigenpair. It concludes that the eigenpair has a separately complex analytic extension to $\mathcal{O}:= \prod_{j\geq 1}\mathcal{O}_j$.
 \end{proof}

  \subsubsection{The bound of parametric derivatives of the ground eigenpairs}
  \label{linbound}

This section is devoted to finding desired bound \eqref{deri-bound} with a proper sequence $\bm{\alpha}$ specified. First of all, with Lemma \ref{LEVP_separate} proven, taking derivatives is justified. Actually, the bound is well proved in \cite{alexey}, so it is sufficient to take the result from \cite{alexey} which is the form of $(|\bm{\nu}|!)\bm{\beta}^{\bm{\nu}}$ because $|\bm{\nu}|!=n!$ if $\bm{\nu}=n\bm{e}_j$ for any $j\geq 1$. Although the bound is already known, the purpose of showing the following proof is to emphasize the advantage of avoiding multi-index argument. As the result of avoiding the multi-index arguments, a different proof than the one in \cite{alexey} is presented.  Before finding the bounds, a technical lemma is presented.  \\
  
  \begin{lemma}
  \label{alpha_lemma}
  Define a sequence $\bm{\alpha}:=(\alpha_n)_{n\geq 0}$ with $\alpha_1=1$ by
  \begin{align}
 \label{alpha}
  \alpha_n
  =
  \sum_{m=1}^{n-1}
  \begin{pmatrix}
  n\\m
  \end{pmatrix}
  \alpha_{n-m}
 \alpha_m.
  \end{align}
  Then, for every $n\geq 1$, $n\alpha_{n-1}\leq \alpha_n$ and 
the sequence satisfies \eqref{radius}.
\end{lemma}

\begin{proof}
First of all, for given $n\geq 1$,
note that the sequence defined by \eqref{alpha} satisfies the following formula
\begin{align}
\label{OESI}
\alpha_n=\sum_{k=1}^{n-1}
\begin{pmatrix}
n\\
k
\end{pmatrix}
\alpha_k
\alpha_{n-k}
=
\frac{(2(n-1))!}{(n-1)!}
\end{align}
for all $n\geq 1$.  The reference to this equality can be found in OESI (A001813), and the name of the sequence is a Quadruple factorial number. This equality can be found in the formula section. Indeed, it can be shown with a simple inductive argument. Firstly, when $n=1$, 
\begin{align}
1=\alpha_1  = \frac{[2(1-1)]!}{(1-1)!}.
\end{align}
Assuming that \eqref{OESI} holds for all $n< N$, observe that
\begin{align}
\sum_{k=1}^{N-1} 
\begin{pmatrix}
N\\
k
\end{pmatrix}
\alpha_k
\alpha_{N-k}
&=
\sum_{k=1}^{N-1} 
\frac{N!}{(N-k)!k!}
\frac{[2(N-k-1)]!}{(N-k-1)!} \frac{[2(k-1)]!}{(k-1)!}
\notag\\
&=N! \sum_{k=1}^{N-1}
\frac{[2(N-k-1)!]}{(N-1-k)!(k-1)!} \frac{[2(k-1)]!}{(k-1)!k!}
\notag\\
&=
N! \sum_{k=1}^{N-1}
C_{N-1-k}C_{k-1}
\notag\\
&=
N! C_{N-1}
\notag\\
&=
N!
\frac{[2(N-1)!]}{(N-1)!N!}
\notag\\
&=
\frac{[2(N-1)]!}{(N-1)!},
\end{align} 
where $C_n$ is $n$th Catalan number(sequence A000108 in the OEIS) and the third, the fourth and the fifth equalities are from the properties of Catalan number. Then observe that
\begin{align}
n\alpha_{n-1}
=
n
\frac{(2(n-2))!}{(n-2)!}
=
\frac{n(n-1)}{(2n-2)(2n-3)}
\frac{(2(n-1))!}{(n-1)!}
=
\frac{n(n-1)}{(2n-2)(2n-3)}
\alpha_n.
\end{align}
With a fundamental computation, it is easy to check
\begin{align}
g(x)=\frac{x(x-1)}{(2x-3)(2x-4)}
\end{align}
is a decreasing function with $g(2)=1$. Therefore, for every $n\geq 2$, $g(n)\leq 1$, so the inequality $n\alpha_{n-1}\leq \alpha_n$ is concluded. Lastly, with the formula \eqref{OESI}, observe that
\begin{align}
\label{vep}
\lim_{n\rightarrow \infty}
\frac{\alpha_n(n+1)}{\alpha_{n+1}}
=
\lim_{n\rightarrow \infty}
\frac{\frac{(2(n-1))!}{(n-1)!}}{\frac{(2n))!}{(n)!}}(n+1)
=
\lim_{n\rightarrow \infty}
\frac{n(n+1)}{(2n)(2n-1)}=\frac{1}{4}>0,
\end{align} 
as desired.
\end{proof}

With the numerical sequence presented above, it is now possible to show the desired bound for the parametric derivatives of the ground eigenpairs.

\begin{lemma}
\label{LEVP_bound}
Suppose $(\lambda(\bm{y}),u(\bm{y}))$ is the solution to the EVP \eqref{problem1} under Assumption \ref{assumption1}. Then, with $\overline{\lambda}=\sup_{\bm{y}\in U} |\lambda(\bm{y})|$ and $\overline{u}=\sup_{\bm{y}\in U}\|u\|_{H_0^1}$, the following bounds hold:
\begin{align}
\label{LEVPbound}
\left| \partial_j^n \lambda(\bm{y}) \right| 
&\leq
\overline{\lambda} \alpha_n  b_j^n,
 \\
 \left\| \partial_j^n u(\bm{y}) \right\|_{H_0^1} 
 &\leq
  \overline{u} \alpha_n b_j^n,
\end{align}
where $\bm{b}:=  \zeta \bm{c} $ with $ c_j:= \max\{\|A_j\|_{L^\infty}, \| B_j  \|_{L^\infty}, \| C_j  \|_{L^\infty}  \}$ and some constant $\zeta>1$.
\end{lemma} 

\begin{proof}
The existence of $\overline{\lambda}$ and $\overline{u}$ is clear due to the compactness of $U$ and the continuous of $(\lambda,u)$ with Heine-Borel Theorem. The first goal is to show 
\begin{align}
\label{lbound}
|\partial_j^n \lambda(\bm{y})|
&\leq  \frac{\overline{\lambda} \alpha_n}{\rho} (\sigma\rho c_j )^n,
\end{align}
and
\begin{align}
\label{ubound}
\left\|\partial_j^n u(\bm{y})\right\|_{H_0^1}
&\leq  \frac{\overline{u} \alpha_n}{\rho} (\sigma\rho c_j )^n,
\end{align}
by the mathematical induction. The constants $\rho$ and $\sigma$ are defined by
\begin{align}
\sigma
:=
\max
\left\{1,
\frac{(2+\overline{\lambda})}{\underline{A}}
,
\left(
2+
\overline{\lambda}
+\frac{ C_{\mu}\overline{u}}{2}
\right)
\right\},
\end{align}
and
\begin{align}
\rho
:=
\max
\left\{
1,
\left(
\frac{(2+\overline{\lambda})}{ \underline{A}}
+
\frac{\overline{\lambda} 
\overline{C}}{\underline{A}}
+
\frac{\overline{\lambda}}{ \underline{A}}
\right)
,
\frac{1}{C_{\mu}}
\left(
\frac{1}{\sigma}
\left(
2
+
\overline{\lambda}
+
\frac{C_{\mu}}{2}
\overline{u}
\right)
+
\left(
\overline{\lambda}
+
\frac{C_{\mu}}{2}
\overline{u}^2
\right)
\left(
\overline{C}
+
\frac{1}{\sigma}
\right)
\right)
\right\},
\end{align}
where the constant $C_\mu>0$ is defined in Lemma 4.2 in \cite{alexey}.
 For simplicity, some preliminary results before the induction process from \cite{alexey} will be taken.
Note that, by Assumption \ref{assumption1}, Assumption 3.1 in \cite{alexey} is satisfies with the coefficient functions $A,B$ and $C$ in \eqref{problem1} with $\frac{1}{R_j}=c_j$ and $\delta=1$. Then note that $\partial^{\bm{\nu}}=\partial^{n\bm{e}_j}=\partial_j^n$. From (5.6) in \cite{alexey}, with Assumption \ref{assumption1}, the following bounds for derivative of eigenvalue $\lambda(\bm{y})$ hold:
\begin{align}
\label{l1}
|\partial_j^n \lambda|
&\leq
\bigg(
n
\|\partial^{n-1}u\|_{H_0^1}
\left(
\|A_j\|_{L^\infty}
+
\|B_j\|_{L^\infty}
+
\lambda \|C_j\|_{L^\infty}
\right)
\notag\\
&+
\sum_{m=1}^{n-1}
\begin{pmatrix}
n\\m
\end{pmatrix}
|\partial^{n-m} \lambda|
\left(
\|\partial^{m} u\|_{H_0^1}
\| C\|_{L^\infty}
+
m\|\partial^{m-1} u\|_{H_0^1}
\| C_j\|_{L^\infty}
\right)
\bigg)\|u\|_{H_0^1}.
\end{align}
For $n=1$, the bound \eqref{l1} is
\begin{align}
|\partial_j \lambda|
&\leq 
\left(
\|u\|_{H_0^1}
\left( \|A_j\|_{L^\infty}+\|B_j\|_{L^\infty}+\overline{\lambda} \|C_j\|_{L^\infty} \right)
\right)\|u\|_{H_0^1}
\notag\\
&
\leq 
\overline{u}^2 (2+\overline{\lambda})c_j
\notag\\
&=
\frac{(2+\overline{\lambda})}{\underline{A}}
\frac{\overline{\lambda} \alpha_1}{\rho}(\rho c_j)
\notag\\
&\leq
\frac{\overline{\lambda} \alpha_1}{\rho}(\rho \sigma c_j),
\end{align}
Assuming \eqref{lbound} and \eqref{ubound} for all derivative order smaller than $n$, \eqref{l1} can be rewritten by 
\begin{align}
\label{l2}
|\partial_j^n \lambda(\bm{y})|
&
\leq
\bigg(
n \frac{\overline{u} \alpha_{n-1}}{\rho} (\sigma\rho c_j)^{n-1}(2+\overline{\lambda})c_j
\notag\\
&+
\frac{1}{\rho^2}
\sum_{m=1}^{n-1}
\begin{pmatrix}
n\\m
\end{pmatrix}
\overline{\lambda}
\alpha_{n-m}
(\sigma\rho c_j)^{n-m}
\left(
\overline{u} \alpha_{m}
(\sigma\rho c_j)^{m} \overline{C}
+
m\overline{u}\alpha_{m-1}
(\sigma\rho c_j)^{m-1} c_j
\right)
\bigg)
\overline{u}
\end{align}
With some algebraic manipulation, the bound in \eqref{l2} can be further simplified to
\begin{align}
\label{l3}
\bigg(
n\alpha_{n-1}
\frac{
(2+\overline{\lambda})
}{\sigma\rho}
+
\frac{\overline{\lambda}\overline{C}}{\rho}
\sum_{m=1}^{n-1}
\begin{pmatrix}
n\\m
\end{pmatrix}
\alpha_{n-m} \alpha_m
 +
\frac{\overline{\lambda}}{\sigma\rho}
 \sum_{m=1}^{n-1}
\begin{pmatrix}
n\\m
\end{pmatrix}
\alpha_{n-m} m\alpha_{m-1}
\bigg)
\frac{\overline{u}^2}{\rho}(\sigma\rho c_j)^n.
\end{align}
By using the fundamental fact about the sequence, $n\alpha_{n-1}\leq\alpha_n$, and the definition of $\alpha_n$ from Lemma \ref{alpha_lemma}, it is  easy to see that
\begin{align}
|\partial_j^n \lambda(\bm{y})|
&\leq 
\left(
\frac{(2+\overline{\lambda})}{ \sigma\underline{A}}
+
\frac{\overline{\lambda} 
\overline{C}}{\underline{A}}
+
\frac{\overline{\lambda}}{\sigma \underline{A}}
\right)
\frac{
\overline{\lambda}
\alpha_n
}{\rho^2}
 (\sigma\rho c_j)^{n}
\leq 
\frac{
\overline{\lambda}
\alpha_n
}{\rho}
(\sigma\rho c_j)^n,
\end{align}
as desired from \eqref{lbound}. Similarly, take the following bound from (5.16) in \cite{alexey} with Assumption \ref{assumption1}:
\begin{align}
\label{u1}
&\|\partial^n_j u \|_{H_0^1}
\leq 
\frac{1}{C_{\mu}}\bigg(
n
\|\partial^{n-1}u\|_{H_0^1}
\left(
\|A_j\|_{L^\infty} 
+
\|B_j\|_{L^\infty}
+
\left(
\overline{\lambda}
+
\frac{C_{\mu}}{2}
\overline{u}
\right)
\|C_j\|_{L^\infty}
\right)+
\notag \\
&
\sum_{m=1}^{n-1}
\begin{pmatrix}
n\\m
\end{pmatrix}
\left(
|\partial^{n-m}\lambda|
+
\frac{C_{\mu}}{2}
\overline{u}
\|\partial^{n-m}u\|_{H_0^1}
\right)
\left(
\|\partial^{m} u\|_{H_0^1}\|C\|_{L^\infty}
+
m
\|\partial^{m-1}u\|_{H_0^1}
\|C_j\|_{L^\infty}
\right)
\bigg).
\end{align}
When $n=1$, bound \eqref{u1} is
\begin{align}
\|\partial_j u\|_{H_0^1}
&\leq 
\frac{1}{C_{\mu}}
\|u\|_{H_0^1}
\left(
\|A_j\|_{L^\infty} 
+
\|B_j\|_{L^\infty}
+
\left(
\overline{\lambda}
+
\frac{C_{\mu}}{2}
\overline{u}
\right)
\|C_j\|_{L^\infty}
\right)
\notag\\
&\leq
\frac{\overline{u}}{C_{\mu}}
\left(
2+
\overline{\lambda}
+\frac{ C_{\mu}\overline{u}}{2}
\right)c_j
\notag\\
&=
\frac{\overline{u}}{\rho}(\sigma \rho c_j).
\end{align}
Now, take the inductive assumptions, \eqref{lbound} and \eqref{ubound}. Then \eqref{u1} is simplified to 
\begin{align}
&
\frac{1}{C_{\mu}}\bigg(
n
\frac{\overline{u}\alpha_{n-1}}{\rho}(\sigma\rho c_j)^{n-1}
\left(
2
+
\left(
\overline{\lambda}
+
\frac{C_{\mu}}{2}
\overline{u}
\right)
\right)
c_j
\notag\\
&+
\sum_{m=1}^{n-1}
\begin{pmatrix}
n\\m
\end{pmatrix}
\alpha_{n-m}
(\sigma\rho c_j)^{n-m}
\left(
\frac{\overline{\lambda}}{\rho}
+
\frac{C_{\mu}}{2\rho }
\overline{u}^2
\right)
\left(
\frac{\overline{u}\alpha_{m}}{\rho}
(\sigma\rho c_j)^{m}
\overline{C}
+
m
\frac{\overline{u}
\alpha_{m-1}}{\rho}
(\sigma\rho c_j)^{m-1}
c_j
\right)
\bigg).
\end{align}
With some algebraic manipulations and with some fundamental facts about the sequence $\alpha_n$ from Lemma \ref{alpha_lemma}, 
\begin{align}
\|\partial_{j}^n u\|_{H_0^1}
&\leq
\frac{1}{C_{\mu}}
\left(
\frac{1}{\sigma}
\left(
2
+
\overline{\lambda}
+
\frac{C_{\mu}}{2}
\overline{u}
\right)
+
\left(
\overline{\lambda}
+
\frac{C_{\mu}}{2}
\overline{u}^2
\right)
\left(
\overline{C}
+
\frac{1}{\sigma}
\right)
\right)
\frac{\overline{u}\alpha_n}{\rho^2} (\sigma\rho c_j)^{n}
\notag\\
&
\leq
\frac{\overline{u}\alpha_n}{\rho}(\sigma\rho c_j)^n.
\end{align}
Therefore, by mathematical induction, the bounds \eqref{lbound} and \eqref{ubound} hold. Furthermore, recalling that $\rho,\sigma\geq 1$,  the bounds in \eqref{LEVPbound} easily follow from \eqref{lbound} and \eqref{ubound}. 
\end{proof}

\subsubsection{The parametric continuity of the ground eigenpairs}
\label{lincontinuity}

This subsection is devoted to showing the continuity of the ground eigenpair on the parameters in $H_{\bm{\gamma}}$.    First, the continuity of the ground eigenpair on coefficient function will be shown. Then, it will be shown that the condition on the sequence $\bm{\gamma}$ in \eqref{Gamm} is enough to ensure that the image of $H_{\bm{\gamma}}$ under the the mapping to the coefficient functions is contained in the region where the ground eigenpair is continuous.  In other words, the mapping $H_{\bm{\gamma}}\rightarrow \mathbb{C}\times H_0^1$ will be considered as a composition of two mapping, parameter to coefficient functions and coefficient functions to the eigenpairs. The following first lemma is about mapping from coefficient functions to the eigenpairs.
\\

\begin{lemma}
\label{LEVP1}
Consider a set $\mathcal{Q}\subset (L^\infty)^3:= L^\infty\times L^\infty \times L^\infty $ defined by
\begin{align}
\label{Q}
\mathcal{Q}:=\{(A,B,C): Re(C)\geq \underline{C}, Re(A)\geq \underline{A} \text{, } \exists (\lambda,u)\in \mathbb{C} \times H_0^1 \text{ }  \text{ satisfying } 
\eqref{nonp}
\},
\end{align} 
where $Re(C)$ and $Re(A)$ is the real parts of the coefficient function $C$ and $A$, and the constants  $\underline{A},\underline{C}>0$ are defined in \eqref{underline}. Then the ground eigenpair $(\lambda,u)\in \mathbb{C}\times H_0^1$ of the following non-parametric problem:
\begin{align}
\label{nonp}
\begin{cases}
-\nabla \cdot (A(x)\nabla)u(x)+B(x)u(x)=\lambda C u(x), &x\in \Omega,\\
u(x)=0, &x\in \partial \Omega,
\end{cases}
\end{align}
continuously depends on $(A,B,C)$ in $ \mathcal{Q}$.
\end{lemma}
\begin{proof}
The implicit function theorem will be used. First, with $E=span_{H_0^1}\{u\}^{\perp}$, define a mapping $\mathcal{N}:\mathcal{Q}\times \mathbb{C}\times E$ by 
\begin{align}
\mathcal{N}(A',B',C',\mu,w) : = (-\nabla\cdot ((A+A')\nabla)+(B+B')-(C+C')(\lambda+\mu))(u+w),
\end{align}
then $\mathcal{N}(0,0,0,0,0)=0$ because $(\lambda,u)$ is the ground-eigenpair of \eqref{nonp}.
The Gateaux derivative of $\mathcal{N}$ with respect to $\mathbb{C}\times E$ is calculated as below: 
\begin{align}
\label{Gderi}
D_{(\mu,w)}\mathcal{N}(0,0,0,0,0)(\mu,w)
&=
\lim_{h\rightarrow 0}
\frac{\mathcal{N}(0,0,0,h\mu,hw)-N(0,0,0,0,0)}{h}
\notag\\
&=
\lim_{h\rightarrow 0}
\frac{(-\nabla\cdot (A\nabla)+B-C(\lambda+h\mu))(u+hw)}{h}
\notag\\
&=
(-\nabla\cdot (A\nabla)+B-C\lambda)w-\mu C  u,
\end{align}
where the second equality is because $\mathcal{N}(0,0,0,0)=0$. Note that the operator $\mathcal{D}:=D_{(\mu,w)}\mathcal{N}(0,0,0,0): \mathbb{C}\times E\rightarrow H^{-1}$ is a bijection because $\mathcal{L}':=(-\nabla\cdot (A\nabla)+B-C \lambda) :E\rightarrow E'$ is bijective where $E'=span_{H^{-1}}\{Cu\}^{\perp}$. In order to see this, recall that (for example, explained in \cite{alex2019}) there is a sequence of eigenpairs $((\lambda_n,u_n))_{n\geq 0}$ such that $0<\lambda_0 \leq \lambda_1 \leq \cdots $ with $(\lambda_0,u_0)=(\lambda,u)$ of problem \eqref{nonp} and the sequence $(u_n)_{n\geq 0}$ is the orthogonal basis of $H_0^1$. Then recalling that $H_0^1$ is dense in $H^{-1}$ and $(u_n)_{n\geq 1}$ is an orthogonal basis of $E$, with the equation $\mathcal{L}'u_n = (\lambda_n-\lambda)Cu_n\in E'$ for all $n\geq 1$, it is clear that $\mathcal{L}'(E)=span\{Cu_n\}=E'$ due to the strict positivity of $C$ in the definition of \eqref{Q}. Since the kernel of $\mathcal{L}'$ is trivial, $\mathcal{L}':E\rightarrow E'$ is bijective. Note that the second term $\mu Cu$ spans $(E')^{\perp}$. Thus the image of $\mathcal{D}$ is $(E')\oplus(E')^\perp = H^{-1}$, and this proves the bijectivity. Also, recall that $\mathcal{L}'$ has continuous inverse $(\mathcal{L}')^{-1}:E'\rightarrow E$ from elementary PDE theory. Now, it is clear to see that $\mathcal{D}^{-1}$ is bounded. Indeed, for any $\psi \in H^{-1}$, there exists $(\mu,w)\in \mathbb{C}\times E$ such that $\mathcal{D}(\mu,w)=\psi$, and observe that, with varying constant $R>0$,
\begin{align}
\label{inversebound}
\|\mathcal{D}^{-1} \psi  \|_{\mathbb{C}\times H_0^1}
&=
|\mu |^2 + \|w\|_{H_0^1}^2
\notag\\
&=
 \frac{1}{\|Cu  \|_{H^{-1}}} \|\mu Cu  \|_{H^{-1}}^2
 + 
 \|w  \|_{H_0^1}^2
\notag \\
&\leq
 \frac{1}{\|Cu  \|_{H^{-1}}} \|\mu u  \|_{H^{-1}}^2
 + 
 R\|\mathcal{L}' w  \|_{H^{-1}}^2
\notag \\
 &
 \leq 
 R\left(  \|\mu C u  \|_{H^{-1}}  
 +
 \|\mathcal{L}' w  \|_{H^{-1}}^2
  \right)
\notag  \\
  &=
  R
  \left(
  \| \mathcal{L}'w - \mu Cu\|_{H^{-1}}^2
  \right)
  \notag\\
  &=
  R\|\mathcal{D}(\mu, w)  \|_{H^{-1}}^2
\notag  \\
  &=
  R\|\psi\|_{H^{-1}}^2,
\end{align}
where the second equality is because $\|Cu(\bm{y})\|_{H^{-1}}\neq 0$ as $\|u(\bm{y})\|_{L^2}=1$ and by the given condition that $Re(C)(x)\geq \underline{C}>0$ for all $x\in \Omega$, the first inequality is because $\mathcal{L}':E\rightarrow E'$ has bounded inverse, the third equality is because $u(\bm{y})$ and $\mathcal{L}'w\in E'$ are orthogonal by the definition of $E'$. Then, by Banach isomorphism theorem, $\mathcal{D}$ is bounded, and so $\mathcal{N}$ is Fr\'echet differentiable. Then, by the implicit function theorem, because $\mathcal{D}^{-1}$ is bounded, there exists an open neighborhood $O_{(A,B,C)}\subset \mathcal{Q}$ of $(A,B,C)$ and a differentiable function $(\mu, w ): O_{(a,b)}\rightarrow \mathbb{C}\times E$ such that, $\mathcal{N}(A',B',C',\mu(A',B',C'),w(A',B',C'))=0$ for all $(A',B',C')\in O_{(A,B,C)}$. Therefore, the ground eigenpair, $(\lambda(A',B',C'),u(A',B',C') )=(\lambda(A,B,C)+\mu(A',B',C'),u(A,B,C)+w(A',B',C'))$ is continuous in $O_{(A,B,C)}$. Repeating the argument above for all $(A, B, C)\in \mathcal{Q}$, it can be concluded that the ground eigenpair $(\lambda, u)$ depends on the coefficient functions $A$, $B$ and $C$ continuously in $\mathcal{Q}$.
\end{proof}

In the next lemma, it will be shown that the image of $H_{\bm{\gamma}}$ under the mapping to coefficient functions is contained in $\mathcal{Q}$.
\begin{lemma}
\label{LEVP_continuous}
Let $\mathcal{V}\subset \mathbb{C}^\mathbb{N}$  be defined by 
  \begin{align}
  \label{Vset}
  \mathcal{V} : =
    \left\{
    \bm{z}\in \mathbb{C}^\mathbb{N}
    :
    (A(\bm{z}),B(\bm{z}),C(\bm{z}))\in \mathcal{Q}
  \right\},
  \end{align}
where $\mathcal{Q}$ is defined as in \eqref{Q}. Then mappings $u:\mathcal{V}\rightarrow H_0^1$ and $\lambda:\mathcal{V}\rightarrow \mathbb{R}$ are jointly continuous. Furthermore, with $\bm{\gamma}$ defined in Assumption \ref{assumption1}, $H_{\bm{\gamma}}\subset \mathcal{V}$.
\end{lemma}
\begin{proof}
Note that the mapping, $(\lambda, u): U\rightarrow \mathbb{C}\times H_0^1$, is a composition of two mappings $\bm{z}\mapsto (A(\bm{z}),B(\bm{z}),C(\bm{z}))$ and the continuous mapping $(A(\bm{z}),B(\bm{z}),C(\bm{z}))\mapsto (\lambda(\bm{z}),u(\bm{z}))$. Note that the mapping $(A,B,C)$ is jointly continuous in $\mathbb{C}^{\mathbb{N}}$, for example, from Example 1.2.2. in \cite{herve}.

 In order to conclude the proof, it is necessary to show that $\bm{\gamma}$ defined in \eqref{Gamm} is enough to make $H_{\bm{\gamma}} \subset \mathcal{V}$ where $\mathcal{V}$ is defined in \eqref{Vset}. In other words, $\bm{\gamma}$ should satisfies $(A(\bm{z}),B(\bm{z}),C(\bm{z}))\in \mathcal{Q}$ for all $\bm{z}\in H_{\bm{\gamma}}$. The sum $\Gamma$ being finite is needed for the condition $(A(\bm{z}),B(\bm{z}),C(\bm{z}))\in (L^\infty)^3$. For example, notating $\bm{z}\in H_{\bm{\gamma}}$ by $\bm{z}=\bm{y}+\bm{x}$ with $\bm{y}\in U$ and
  \begin{align}
 y_j=\argmin_{y\in [-1,1]} \left| y_j-z_j  \right|,
 \end{align}
for each $j\geq 1$, note that
 \begin{align}
 \left\| A(\bm{z}) \right\|_{\infty}
 &\leq \|A(\bm{y})  \|_\infty  +\sum_{j\geq 1} |x_j|\|A_j\|_{\infty}
 \notag\\
 &\leq \overline{A}  +\sum_{j\geq 1} c_j\left( \frac{\varepsilon}{\gamma_j b_j}  \right)
 \notag\\
 &\leq 
 \overline{A}  +\frac{\varepsilon}{\zeta}\sum_{j\geq 1} \left( \frac{1}{\gamma_j}  \right)<\infty.
 \end{align}

In the rest of the argument, the size of $\Gamma$ will be determined.   Let's set $\bm{\gamma}=(1/\Gamma)\bm{\kappa}$ with 
\begin{align}
\label{K}
K=\sum_{j\geq 1}\frac{1}{\kappa_j}=1,
\end{align}
for the convenience of the analysis. In the further analysis, the size of $\Gamma$ will be determined by the positivity condition for $A$ and $C$, and the existence of the ground eigenpairs with complex perturbation of coefficient functions.

First, as mentioned in \eqref{Q}, $\bm{\gamma}$ should satisfy $A(\bm{x})>\underline{A}$ and $C(\bm{x})\geq \underline{C}$ for all $\bm{z}\in H_{\bm{\gamma}}$ where $\bm{z}=\bm{x}+i\bm{y}$ with $\bm{x},\bm{y}\in \mathbb{R}^\mathbb{N}$.  A simple computation yields
\begin{align}
\label{underC}
C(\bm{x}) &= C_0 + \sum_{j\geq 1} x_j C_j
\notag \\
&\geq 
C_0 - \sum_{j\geq 1} |x_j| c_j
\notag \\
&\geq C_0 
-
\sum_{j\geq 1} 
\left(1+\frac{\varepsilon}{\gamma_j b_j}  \right) c_j
\notag \\
&\geq \underline{C} -\sum_{j\geq 1}\frac{\varepsilon}{\gamma_j \zeta}
\notag \\
&\geq \underline{D} -\sum_{j\geq 1}\frac{\varepsilon}{\gamma_j },
\end{align}
where $\underline{D}:= \min\{\underline{A},\underline{C} \}$. Same argument also shows $A(\bm{x})$ can be lower bounded with exactly same number as \eqref{underC}. For the desired condition, the lower bound above should be positive real number. In other words, for the positivity of $A$ and $C$, it is enough to have
\begin{align}
\label{r-bound}
\Gamma<\frac{\underline{D}}{\varepsilon}=4\underline{D},
\end{align}
which is clearly satisfied by $\bm{\gamma}$ in Assumption \ref{assumption2}.

Next, the number $\Gamma$ should be further specified for the existence of the ground eigenpair in complex perturbation. To be specific, the main concern here is the existence of the ground eigenpair with complex coefficient functions. For this, a variation of the problem \eqref{problem1} is considered. To be specific, in \eqref{problem1}, replace the parameter $\bm{y}$ with
\begin{align}
\bm{y}_{\bm{r}} := 
\bm{y}
+
 \left(\frac{\varepsilon}{\kappa_1 \beta_1}r_1, \frac{\varepsilon}{\kappa_2 \beta_2}r_2,\dots  \right),
\end{align}
for the variable $\bm{r}\in \mathbb{C}^\mathbb{N}$ with the restriction $|r_\ell|=|r_j|$ for all $\ell,j\geq 1$. For convenience, let $r_j=re^{i\theta_j}$ for $\theta_j\in \mathbb{R}$. Here, $\frac{\varepsilon}{\gamma_j \beta_j}$ is from the definition of $H_{\gamma}$ which is explained in the paragraph under \eqref{radi}. The sequence $\bm{\beta}$ is determined in Lemma \ref{LEVP_bound}. The largest possible $r$ that guarantees the existence of the ground eigenpair for $(A(\bm{y}_{\bm{r}}),B(\bm{y}_{\bm{r}}),C(\bm{y}_{\bm{r}}))$ will be the largest possible value of $\Gamma$.

With fixed $\bm{\kappa}$ in \eqref{K}, for fixed $\bm{y}\in U$, the problem \eqref{problem1} is now PDE with one parameter $r $ as follows:
\begin{align}
-\nabla\cdot [A(r)\nabla ]u(r) +B(r)u(r)=C(r)\lambda(r)u(r),
\end{align}
with
\begin{align}
A(r)&= A(\bm{y}) +r \left(\sum_{j\geq 1} \frac{\varepsilon}{\kappa_j \beta_j}e^{i\theta_j} A_j  \right),
\\
B(r)&= B(\bm{y}) +r \left(\sum_{j\geq 1} \frac{\varepsilon}{\kappa_j \beta_j}e^{i\theta_j} B_j  \right),
\\
C(r)&= C(\bm{y}) +r\left(\sum_{j\geq 1} \frac{\varepsilon}{\kappa_j \beta_j}e^{i\theta_j} C_j  \right).
\end{align}
Note that, $A(r),B(r),C(r)\in L^\infty$ because
\begin{align}
&\max\left\{ \left\| A(r)\right\|_{\infty},\left\| B(r)\right\|_{\infty},\left\| C(r)\right\|_{\infty}   \right\}\\
& \leq 
\max\left\{ \left\| A_0\right\|_{\infty},\left\| B_0\right\|_{\infty},\left\| C_0\right\|_{\infty}   \right\}+ r\sum_{j\geq 1} \frac{\varepsilon}{\kappa_j \beta_j} c_j
\\
&= \max\left\{ \left\| A_0\right\|_{\infty},\left\| B_0\right\|_{\infty},\left\| C_0\right\|_{\infty}   \right\}+ \frac{r\varepsilon}{\zeta}\sum_{j\geq 1} \frac{1}{\kappa_j} <\infty, 
\end{align}
for every $r>0$.
With this transformed problem, main interest is in finding the size of $\Gamma$ that guarantees the existence continuous ground eigenpair. A similar argument with Lemma \ref{LEVP_bound} yields
\begin{align}
\left| \partial_r^n \lambda(\bm{y}) \right| 
&\leq
\overline{\lambda} \alpha_n  g^n,
 \\
 \left\| \partial_r^n u(\bm{y}) \right\|_{H_0^1} 
 &\leq
  \overline{u} \alpha_n  g^n,
\end{align}
where $g:=  \zeta h $ with 
\begin{align}
h:=\frac{\varepsilon}{\zeta}\sum_{j\geq 1} \frac{1}{\kappa_j }=\frac{\varepsilon }{\zeta}.
\end{align}
Then, again with the Taylor expansion,
\begin{align}
u(r)
=
\sum_{n=0}^{\infty}
\frac{\partial_r^n u(0)}{n!}r^n,
\end{align}
the following estimate can be obtained:
\begin{align}
\|u(r)\|_{H_0^1}
\leq 
\sum_{n=0}^{\infty}
\frac{\overline{u} \alpha_n  g^n}{n!}|r|^n.
\end{align}
Now, the radius of the convergence is
\begin{align}
\label{r-bound2}
\lim_{n\rightarrow \infty}
\frac{\frac{\overline{u} \alpha_{n+1}  g^{n+1}}{(n+1)!}|r|^{n+1}}{\frac{\overline{u} \alpha_n  g^n}{n!}|r|^n} =\frac{g}{\varepsilon}|r|<1 \implies |r|<\frac{\varepsilon}{g}=1.
\end{align}
Thus, $r$ should be less than $1$ and it is clearly satisfied by the $\Gamma$ in Assumption \ref{assumption1}. Therefore, it can be concluded that $H_{\bm{\gamma}}\subset \mathcal{V}$. 
\end{proof}

From the lemmata above Corollary \ref{LEP} has been proven. To be specific, with the framework suggested at the end of Chapter 2, the first item is shown in Section \ref{linseparate}, the second item is shown in Section \ref{linbound} and the third item is shown in Section \ref{lincontinuity}.  Summing up, using Theorem \ref{main}, it has been shown that the ground eigenpair, $(u(\bm{y}),\lambda(\bm{y}))$, is $\left(\bm{b},\varepsilon\right)-$holomorphic with $\bm{b}=\bm{\gamma \beta}$ and $\varepsilon=\frac{1}{4}$ is from \eqref{vep}. It implies, as a corollary of Theorem \ref{impb}, the following holds:
\\

\begin{corollary}
\label{imp-bound123}
Suppose that $(\lambda(\bm{y}),u(\bm{y}))$ is the ground eigenpair of the problem \eqref{problem1} under Assumption \ref{assumption1}. Then there exists $M_u$ and $M_\lambda$ such that the following bounds hold:
\begin{align}
\label{LEVP-bound}
|\partial^{\bm{\nu}} \lambda (\bm{y})|\leq M_{\lambda} \bm{\nu}! \left(4{\bm{\gamma}\bm{\beta}}\right)^{\bm{\nu}},
\\
\|\partial^{\bm{\nu}} u (\bm{y})\|_{H_0^1}\leq M_{u}\bm{\nu}! \left(4{\bm{\gamma}\bm{\beta}}\right)^{\bm{\nu}},
\end{align}
for every $\bm{\nu}\in \mathcal{F}$ and every $\bm{y}\in U$.
\end{corollary}

\subsection{Semilinear Eigenvalue Problem}

\label{semilinear} The following is the parametric semilinear eigenvalue problem of interest:
 \begin{align}
 \label{problem2}
 \begin{cases}
 -\nabla\cdot(A(x)\nabla )u(x,\bm{y})+{B(x,\bm{y})}u({x},\bm{y})+\eta u({x},\bm{y})^p=\lambda u({x},\bm{y}), &({x},\bm{y})\in \Omega\times U,\\
u({x},\bm{y}) =0, &({x},\bm{y})\in \partial \Omega \times U,
 \end{cases}
 \end{align}
 where $\Omega\subset \mathbb{R}^d$ is a bounded open domain with $C^2$ boundary and $\eta>0$ is a constant. When $A\equiv 1$, $\partial\Omega$ satisfying Lipschitz boundary is allowed. In this problem, the coefficient $A$ does not depend on the parameter $\bm{y}$. For this problem, the following condition will be assumed:\\

 \begin{assumption}
 \label{assumption2}
\begin{itemize}
\item There exists constants $\underline{A}>0$ such that $A(x)\geq \underline{A}$ for all $x\in \Omega$,

\item $B(x,\bm{y})$ has the following form:
 \begin{align}
B(x,\bm{y})=B_0+\sum_{j\geq 1}y_j B_j(x),
 \end{align}
where $ (B_j)_{j\geq 0}\subset L^\infty$,

\item The pair $(d,p)$ belongs to $\mathcal{A}$, where 
\begin{align}
\label{dp}
\mathcal{A}&:=\left\{ (d,p)\in \mathbb{N}\times \mathbb{N}: (d,p)\in [1,2]\times [1,\infty)\text{ or } p\in \left[1, \frac{d}{d-2}  \right] \right\},
\\
\label{dp2}
\mathcal{A}'&:=\left\{ (d,p)\in \mathbb{N}\times \mathbb{N}: (d,p)\in [1,2]\times [1,\infty)\text{ or } p\in \left[1, \frac{2}{d-2}  \right] \right\}.
\end{align}

\item There exists $\bm{\gamma}\in (1,\infty)^\mathbb{N}$ such that  
\begin{align}
\label{gammasemi}
\Gamma : = \sum_{j\geq 1} \frac{1}{\gamma_j}<1,
\end{align}
where $\beta_j:=\rho\|B\|_{L^\infty}$ and $\rho> 1$ is determined as in Lemma \ref{semi2}.

\end{itemize}
\end{assumption}
This EVP has received great attention from physics. In particular, when $p=3$, this problem recovers Gross-Pitaevskii equation which is known to describe super-fluidity and super-conductivity; for example, see \cite{gpe, gpe2}. The parametric version of this problem was recently studied in \cite{Bahn2024}, and, in the paper, an estimation for the upper bound for the mixed derivatives that is equivalent with \cite{alexey} is obtained. The analytic dependence of the ground eigenpair near $U$ is also shown in \cite{Bahn2024} with the use of implicit function theorem to support the meaning of taking any finitely high order mixed derivatives. In this section, the result from \cite{Bahn2024} will be directly used to show $(\bm{b},\varepsilon)$-holomorphy of the ground eigenpair of this problem.
\\

\begin{corollary}
\label{SLEP}
Consider the ground eigenpair $(u(\bm{y}),\lambda(\bm{y}))$ of the semilinear eigenvalue problem \eqref{problem2} under Assumption \ref{assumption2}.  Then the eigenpair is $(\bm{ b},\varepsilon)-$holomorphic with $\bm{b}=\bm{\gamma \beta}$, $\beta_j:=\rho \|B_j\|_{L^\infty}$ for some constant $\rho>1$ and with $\varepsilon=1$.
\end{corollary}

The proof of Corollary \ref{SLEP} consists of the following lemmata. First separately holomorphic property is stated.
\\

\begin{lemma}
\label{SEVP2}
The ground eigenpair $(u(\bm{y}),\lambda(\bm{y})):U\rightarrow H_0^1\times \mathbb{R}$ of \eqref{problem2} has separately holomorphic extension onto $\mathcal{O}$ where $U\subset \mathcal{O}\subset \mathbb{C}^\mathbb{N}$ with $\mathcal{O}:=\prod_{j\geq 1}\mathcal{O}_j $ where  $\mathcal{O}_j\subset \mathbb{C}$ is an open set for all $j\geq 1$.
\end{lemma}
\begin{proof}
The existence of the holomorphic extension of the ground eigenpair onto $\mathcal{O}\subset \mathbb{C}^\mathbb{N}$ has been shown in Theorem 3.1 in \cite{Bahn2024}.
\end{proof}

In the next lemma, the bound of the derivative is stated.

 \begin{lemma}
 \label{semi2}
 Let $(u(\bm{y}),\lambda(\bm{y}))$ be the ground eigenpair of \eqref{problem2}. Then, for any $j\geq 1$, any $n\geq 1$ and any $\bm{y}\in U$, the following bound for the derivatives holds:
 \begin{align}
 \|\partial_j^n u(\bm{y}) \|_{H_0^1}
 &\leq 
 \overline{u} (n!) \beta_j^n
 \\
  |\partial_j^n \lambda(\bm{y}) |
 &\leq 
 \overline{\lambda} (n!) \beta_j^n
 \end{align}
 where $\bm{\beta}=\rho \|B_j\|_{L^\infty}$ with $\rho>1$ is some large enough constant.
 \end{lemma}
\begin{proof}
In Theorem 3.2 of \cite{Bahn2024}, it is shown that, for $\bm{\nu}\in \mathcal{F}$ and for every $\bm{y}\in U$, the following bounds hold:
\begin{align}
\left\|  \partial^{\bm{\nu}}u(\bm{y}) \right\|_{H_0^1} &\leq \overline{u} (|\bm{\nu}|!)\bm{\beta^\nu},
\\
| \partial^{\bm{\nu}}\lambda (\bm{y})  |&\leq \overline{\lambda}(|\bm{\nu}|!) \bm{\beta^\nu}.
\end{align}
With $\bm{\nu}$ such that $\nu_j=n$ and $\nu_i=0$ for all $i\neq j$, the notations are simplified as $|\bm{\nu}|=n$ and $\bm{\beta}^{\bm{\nu}}=\beta_j^n$ which is the desired result.
\end{proof}

Additionally, note that the sequence $\bm{\alpha}$ in Theorem \ref{main} in this case is define by  $\alpha_n=n!$, and so
\begin{align}
\varepsilon=\lim_{n\rightarrow \infty}\frac{\alpha_n(n+1)}{\alpha_{n+1}}=\lim_{n\rightarrow \infty}=\frac{(n+1)!}{(n+1)!}=1.
\end{align}

In the next lemma, the continuous dependence of the ground eigenpair on the potential coefficient function $B$ is stated.

\begin{lemma}
\label{SEVP1}
For given $A\in L^\infty$ satisfying the condition in Assumption \ref{assumption2} and given $\eta >0$, consider a set $\mathcal{Q}'\subset L^\infty$ defined by
\begin{align}
\label{Q'}
\mathcal{Q}':=\{B\in L^\infty: \exists (\lambda,u)\in \mathbb{C} \times H_0^1 \text{ }  \text{ satisfying  \eqref{nonp}
}  
\}.
\end{align} 
Then the ground eigenpair $(\lambda,u)\in \mathbb{C}\times H_0^1$ of the following non-parametric problem:
\begin{align}
\label{nonpsemi}
\begin{cases}
-\nabla \cdot (A(x)\nabla)u(x)+B(x)u(x)+\eta (u(x))^3=\lambda u(x), &x\in \Omega,\\
u(x)=0, &x\in \partial \Omega,
\end{cases}
\end{align}
continuously depends on $B$ in $ \mathcal{Q}'$.
\end{lemma}
\begin{proof}
The proof of this lemma is found in Lemma 3.3 in \cite{Bahn2024}. 
\end{proof}

Lastly, it will be shown that the sequence $\bm{\gamma}$ chosen in Assumption \ref{assumption2} is enough to ensure the continuity of the ground eigenpair in $H_{\bm{\gamma}}$.

\begin{lemma}
Let $\mathcal{V}\subset \mathbb{C}^\mathbb{N}$  be defined by 
  \begin{align}
  \label{Vsetsemi}
  \mathcal{V} : =
    \left\{
    \bm{z}\in \mathbb{C}^\mathbb{N}
    :
    B(\bm{z}) \in \mathcal{Q}'
  \right\},
  \end{align}
where $\mathcal{Q}'$ is defined as in \eqref{Q'}. Then mappings $u:\mathcal{V}\rightarrow H_0^1$ and $\lambda:\mathcal{V}\rightarrow \mathbb{R}$ are jointly continuous. Furthermore, with $\bm{\gamma}$ defined in Assumption \ref{assumption2}, $H_{\bm{\gamma}}\subset \mathcal{V}$ which implies the jointly continuity of the ground eigenpair on $H_{\bm{\gamma}}$.
\end{lemma}

\begin{proof}

As mentioned in Lemma \ref{LEVP_continuous},  the coefficient function is jointly continuous in $\mathbb{C}^{\mathbb{N}}$. Also, from the previous lemma, the eigenpair, $(u,\lambda):\mathcal{Q}'\rightarrow H_0^1\times \mathbb{C} $ is continuous. Therefore, it is left to show that the sequence $\bm{\gamma}$ in Assumption \ref{assumption2} is enough to have $H_{\bm{\gamma}}\subset \mathcal{V}$.

First of all, as it has been done in linear case, for fixed $\bm{y}\in U$ with $\bm{\gamma}:=(1/\Gamma)\bm{\kappa}$, define
\begin{align}
B(r):=B(\bm{y})+\left( \sum_{j\geq 1} \frac{\varepsilon}{\kappa_j \beta_j}r_j B_j \right),
\end{align} 
for complex variable $\bm{r}\in \mathbb{C}^\mathbb{N}$. Also, following the notation from the linear case, let $r_j=r e^{i\theta_j}$ for all $j\geq 1$. In other words, the main focus here is to find how large $|r|$ could be to ensure that the ground eigenpair exists with the corresponding coefficient function $B(r):=B(\bm{z})$ for
\begin{align}
\bm{z}=\bm{y}+r \left( \frac{\varepsilon}{\kappa_1 \beta_1}e^{i\theta_1},\frac{\varepsilon}{\kappa_2 \beta_2}e^{i\theta_2},\dots  \right).
\end{align}
The largest possible $r$ will be the largest possible $\Gamma$. With same argument made in the proof of Lemma \ref{semi2}, it is possible to show that
\begin{align}
\|\partial^n u(r)\|_{H_0^1} \leq \overline{\lambda} n! (\rho g)^n
\end{align}
where
\begin{align}
g:=\frac{\varepsilon}{\rho} =
 \frac{\varepsilon}{\rho}\sum_{j\geq 1}\frac{1}{\kappa_j}
 \geq 
  \sum_{j\geq 1} \frac{\varepsilon}{\kappa_j \beta_j} \|B_j\|_{L^\infty}
  \geq 
\left\| \sum_{j\geq 1} \frac{\varepsilon}{\kappa_j \beta_j} B_j\right\|_{L^\infty} .
\end{align}

Then, the radius of the converge of the series
\begin{align}
u(r) = \sum_{n\geq 0} \frac{\partial^n u(r) }{n!}r^n
\end{align}
with some computations,
\begin{align}
\|u(r)\|_{H_0^1} \leq  \sum_{n\geq 0} \frac{n!(\rho g)^n}{n!}|r|^n
=
\sum_{n\geq 0}
(\rho g)^{n}|r|^n
,
\end{align}
is
\begin{align}
\lim_{n\rightarrow}\frac{(\rho g)^{n+1}|r|^{n+1}}{(\rho g)^{n}|r|^n}
=
(\rho g)|r|<1 
\implies
  |r|<\frac{1}{\rho g}. 
\end{align}

It implies that
\begin{align}
|r|<\frac{1}{\rho } \frac{\rho }{\varepsilon }=1\implies \Gamma <1.
\end{align}

This condition is true thanks to \eqref{gammasemi}. Therefore, the ground eigenpair exists in $H_{\bm{\gamma}}$ with any $\bm{\gamma}$ satisfying \eqref{gammasemi}.

\end{proof}

Finally, as a corollary of Thoerem \ref{impb}, by Lemma \ref{SEVP2}, Lemma \ref{semi2} and Lemma \ref{SEVP1},  the following bound for mixed derivatives hold:
\\

\begin{corollary}
\label{mixed2}
Suppose $(\lambda(\bm{y}),u(\bm{y}))$ is the ground eigenpair of the problem \eqref{problem2} under Assumption \ref{assumption2}. Then there exists $M_u$ and $M_\lambda$ such that the following bounds hold:
\begin{align}
|\partial^{\bm{\nu}} \lambda (\bm{y})|\leq M_{\lambda} \bm{\nu}! ( \bm{\gamma \beta})^{\bm{\nu}},
\\
\|\partial^{\bm{\nu}} u (\bm{y})\|_{H_0^1}\leq M_{u} \bm{\nu}! ( \bm{\gamma \beta})^{\bm{\nu}}
\end{align}
for every  $\bm{\nu}\in \mathcal{F}$ and every $\bm{y}\in U$, where the sequence $\bm{\gamma \beta}$ is defined as in Assumption \ref{assumption2}.
\end{corollary}

\section{Discussion on quasi-Monte Carlo methods}

 In this section, a short discussion on the application of the mixed derivatives obtained in Corollary \ref{imp-bound123} and Corollary \ref{mixed2} to quasi-Monte Carlo methods(QMC). In this application, the parametric PDEs studied in the previous sections are considered as stochastic PDEs by considering each parameter $y_j$ as a uniform random variable over $\left[-\frac{1}{2},\frac{1}{2} \right]$ for each $j\geq 1$. Thus, the parameter space of interest is changed to $\left[-\frac{1}{2},\frac{1}{2} \right]^\mathbb{N}$ from $[-1,1]^\mathbb{N}$. This this perspective, still denoting the random eigenpair  by $(u(\bm{y}),\lambda(\bm{y}))$, the quantities of our interest will be their expected values:
\begin{align}
\mathbb{E}_{\bm{y}}\left[ \lambda(\bm{y})  \right]=\int_{\left[-\frac{1}{2},\frac{1}{2} \right]^{\mathbb{N}}} \lambda(\bm{y})d\bm{y}=\lim_{s\rightarrow \infty} \int_{\left[-\frac{1}{2},\frac{1}{2} \right]^s} \lambda(y_1,y_2,\dots, y_s,0,\dots)dy_1 dy_2\cdots dy_s,
\end{align}
and
\begin{align}
\mathbb{E}_{\bm{y}}\left[ \mathcal{G}(u(\bm{y}))  \right]=\int_{\left[-\frac{1}{2},\frac{1}{2} \right]^{\mathbb{N}}} \mathcal{G}(u(\bm{y}))d\bm{y}=\lim_{s\rightarrow \infty} \int_{\left[-\frac{1}{2},\frac{1}{2} \right]^s} \mathcal{G}(u(y_1,y_2,\dots, y_s,0,\dots))dy_1 dy_2\cdots dy_s.
\end{align}
for any $\mathcal{G}\in H^{-1}$. QMC method is a variant of Monte Carlo method for faster convergence by enforcing random samples to be uniform throughout the space preventing rare situation such as all uniform random samples are found in a local region. Also, this method is useful for random number generators that are not well verified to be uniform in high dimension. For example, recall that, in original Monte Carlo method, the expectation of a function $f:\Omega\rightarrow \mathbb{R}$ with respect to a given  probability measure $\mathbb{P}$ is estimated by 
\begin{align}
\int_\Omega f(X)d\mathbb{P}(X) = \frac{1}{N}\sum_{i=1}^{N} f(X_i),
\end{align}
where $\Omega \subset \mathbb{R}^s$,  for a positive integer $s$, is a compact subset and $X_i \sim \mathbb{P}$ are iid samples. 

 In QMC, the random samples are generated near the uniformly spread lattice points.  \textit{CBC generated randomly shifted rank 1 lattice rule} is constructed with a generating vector $\bm{z}\in \mathbb{Z}^s$ and with  a random shift $\bm{\Delta}\sim Unif([0,1]^s)$. To be specific, with the defined notations, the ramdom samples are $\left( \left\{\frac{k\bm{z}}{N} +\bm{\Delta} \right\}-\bm{\frac{1}{2}} \right)_{k=0}^{N-1}$ where $\{\}$ notation means taking the fraction part of it, i.e., $\{\frac{3}{2}\}=\frac{1}{2}$, and $\bm{\frac{1}{2}}=\left(\frac{1}{2},\dots,\frac{1}{2} \right)\in \mathbb{R}^s$. Thus, after the dimension truncation, the quantity
 \begin{align}
 \mathbb{E}_{\bm{y}_s}[f_s(\bm{y_s})] : = \int_{\left[-\frac{1}{2},\frac{1}{2} \right]^s} f(\bm{y}_s)d\bm{y}_s,
 \end{align}
is estimated by
\begin{align}
Q_{N,s} f := \frac{1}{N}\sum_{k=0}^{N-1}f\left( \left\{\frac{k\bm{z}}{N}+\bm{\Delta}  \right\}-\bm{\frac{1}{2}} \right).
\end{align}
 In this QMC rule, $\bm{z}$ is constructed using the Fast component-by-component(CBC) algorithm. detailed explanation of this method can be found in \cite{quasi}.
\\

 Recently in \cite{Bahn2024}, combining the results from Theorem 4.2 in \cite{alex2019} and Theorem 6.4 in \cite{sch1} the following general lemma is stated.

\begin{lemma}[Lemma 4.3 in \cite{Bahn2024}]
\label{gtotal}
Let $f: \left[-\frac{1}{2},\frac{1}{2} \right]^\mathbb{N}\rightarrow \mathcal{X}$ be given where $\mathcal{X}$ is a given normed vector space with norm $\|\cdot\|_\mathcal{X}$. Suppose that $f$ is analytic for any of its restrictions to a finite-dimensional domain and that $\|\partial^{\bm{\nu}}f(\bm{y})\|_\mathcal{X}\leq C(|\bm{\nu}|!)\bm{b}^{\bm{\nu}}$ for some constant $C>0$, for some $\varepsilon\in [0,1)$ and for some decreasing sequence $\bm{b}\in \ell^q$.   Then, for sufficiently large $s\in \mathbb{N}$ and for any prime number $N\in \mathbb{N}$, there is a generating vector $\bm{z}\in \mathbb{N}^s$ such that, for any $\mathcal{G}\in X^*$ , 
\begin{align}
\sqrt{
\mathbb{E}_{\bm{\Delta}}
\left[
\left|
\mathbb{E}_{\bm{y}}[\mathcal{G}(f)]-Q_{N,s}\mathcal{G}(f_s)
\right|^2
\right]
}
\leq 
C_\alpha \left( T(s)+N^{-\alpha} \right),
\end{align}
where  
\begin{align}
T(s)=
\begin{cases}
s^{-\frac{2}{q}+1}, &q\in (0,1),\\
\left(\sum_{j>s} b_j \right)^2, &q=1,
\end{cases}
\end{align}
and
\begin{align}
\alpha=
\begin{cases}
1-\delta, &q\in \left(0,\frac{2}{3} \right],\\
\frac{1}{q}-\frac{1}{2}, &q\in \left( \frac{2}{3},1 \right],
\end{cases}
\end{align}
for arbitrary $\delta\in \left(0,\frac{1}{2} \right)$ and for some $s$ independent constants $C_{\alpha}>0$. When $q=1$, we additionally assume that $\sum_{j\geq 1}b_j<\sqrt{6}$.
\end{lemma}

With this lemma and Corollary \ref{imp-bound123} and corollary \ref{mixed2}, the following corollary provides the application to QMC : 
\begin{corollary}
Suppose $(u,\lambda)$ is either the ground eigenpair of \eqref{problem1} under Assumption \ref{assumption1} or the ground eigenpair of \eqref{problem2} under Assumption \ref{assumption2}.
Further suppose that  $\left(  \gamma_j\|b_j\|_{L^\infty}  \right)_{j\geq 1}\in \ell^q$ for some $q\in (0,1]$,  $N\in \mathbb{N}$ is a prime and $\mathcal{G}\in H^{-1}$. Then a lattice rule generating vector $\bm{z}\in \mathbb{N}^s$ can be constructed by CBC algorithm such that the root-mean-square errors of  the CBC-generated randomly shifted rank 1 lattice rule estimation of $\mathbb{E}_{\bm{y}}[\lambda]$ and $\mathbb{E}_{\bm{y}}[\mathcal{G}(u)]$ for any $\mathcal{G}\in H^{-1}$ satisfies
\begin{align}
\sqrt{\mathbb{E}_{\Delta}\left[
\left| \mathbb{E}_{\bm{y}}[\lambda]-Q_{N,s}\lambda_s
\right|^2  \right]}\leq C_{1,\alpha}\left(T(s)+ N^{-\alpha}\right),
\end{align}
and
\begin{align}
\sqrt{\mathbb{E}_{\Delta}
\left[ \left| 
\mathbb{E}_{\bm{y}}[\mathcal{G}(u)]-Q_{N,s}\mathcal{G}(u_s) \right|^2
\right]
}\leq C_{2,\alpha}\left(T(s)+ N^{-\alpha}\right),
\end{align}
where $T(s)$ is defined in Lemma \ref{gtotal} and
\begin{align}
\alpha=
\begin{cases}
1-\delta, &q\in \left(0,\frac{2}{3} \right],\\
\frac{1}{q}-\frac{1}{2}, &q\in \left( \frac{2}{3},1 \right],
\end{cases}
\end{align}
for arbitrary $\delta\in \left(0,\frac{1}{2} \right)$ and for some $s$ independent constants $C_{1,\alpha},C_{2,\alpha}>0$.
\end{corollary}
\begin{proof}
With Lemma \ref{gtotal}, it is enough to note that with $\bm{b}=\bm{\gamma b}\in \ell^{q}$ with $q\in (0,1]$ assumed in Assumption \ref{assumption1} or Assumption \ref{assumption2} and the bound obtained in Corollary \ref{imp-bound123} and Corollary \ref{mixed2} satisfy the desired bound.
\end{proof}

\section{Acknowledgement} The author thanks professor Yulong Lu (University of Minnesota Twin Cities) for suggesting this problem. The author also thanks professor Christoph Schwab (ETH Z\"uich), and professor Juli\'an L\'opez-Gomez (Complutense University of Madrid) for all the helpful discussions.

\section{Data Availability Statement}
Data sharing not applicable to this article as no datasets were generated or analysed during the current study.

\end{document}